\documentclass[a4]{amsart}

\input xypic 
\input xy 
\xyoption{all} 
\usepackage{amssymb} 
\usepackage{bbm}
\usepackage{tikz-cd}

\usepackage{float}
\usepackage{graphicx}
\usepackage{accents}
\usepackage[bookmarksnumbered, bookmarksopen, 
colorlinks,citecolor=blue,linkcolor=blue,backref]{hyperref}

\newtheorem{theorem}{Theorem}[section]
\newtheorem{proposition}[theorem]{Proposition}
\newtheorem{lemma}[theorem]{Lemma}
\newtheorem{corollary}[theorem]{Corollary}

\theoremstyle{definition}

\newtheorem{remark}[theorem]{Remark}
\newtheorem{example}[theorem]{Example}

\newcounter{RomanNumber}
\newcommand{\MyRoman}[1]{\setcounter{RomanNumber}{#1}\Roman{RomanNumber}}

\newcommand{\conn}{\ensuremath{\#}} 
\newcommand{\cp}{\ensuremath{\mathbb{C}P}} 
\newcommand{\hp}{\ensuremath{\mathbb{H}P}} 
\newcommand{\fp}{\ensuremath{\mathbb{F}P}} 
\newcommand{\Gtau}{\ensuremath{\mathcal{G}^{\tau}}} 
\newcommand{\Gzero}{\ensuremath{\mathcal{G}^{0}}} 

\newcounter{bean}
\newenvironment{letterlist}{\begin{list}{\rm ({\alph{bean}})}
      {\usecounter{bean}\setlength{\rightmargin}{\leftmargin}}}
      {\end{list}}

\newcommand{\namedright}[3]{\ensuremath{#1\stackrel{#2}
 {\longrightarrow}#3}}
\newcommand{\nameddright}[5]{\ensuremath{#1\stackrel{#2}
 {\longrightarrow}#3\stackrel{#4}{\longrightarrow}#5}}
\newcommand{\namedddright}[7]{\ensuremath{#1\stackrel{#2}
 {\longrightarrow}#3\stackrel{#4}{\longrightarrow}#5
  \stackrel{#6}{\longrightarrow}#7}}

\newcommand{\larrow}{\relbar\!\!\relbar\!\!\rightarrow}
\newcommand{\llarrow}{\relbar\!\!\relbar\!\!\larrow}

\newcommand{\lnameddright}[5]{\ensuremath{#1\stackrel{#2}
 {\larrow}#3\stackrel{#4}{\larrow}#5}}

\newcommand{\llnameddright}[5]{\ensuremath{#1\stackrel{#2}
 {\llarrow}#3\stackrel{#4}{\llarrow}#5}}

\newcommand{\qqed}{\hfill\Box}

\begin{document}


\title{Stabilization of Poincar\'{e} duality complexes and homotopy gyrations} 

\author{Ruizhi Huang} 
\address{State Key Laboratory of Mathematical Sciences \& Institute of Mathematics, Academy of Mathematics and Systems Science, 
   Chinese Academy of Sciences, Beijing 100190, China} 
\email{huangrz@amss.ac.cn} 
   \urladdr{https://sites.google.com/site/hrzsea/} 
\author{Stephen Theriault}
\address{Mathematical Sciences, University of Southampton, Southampton 
   SO17 1BJ, United Kingdom}
\email{S.D.Theriault@soton.ac.uk}



\begin{abstract} 
Stabilization of manifolds by a product of spheres or a projective space is important in geometry. 
There has been considerable recent work that studies the homotopy theory of stabilization 
for connected manifolds. This paper generalizes that work by developing new methods 
that allow for a generalization to stabilization of Poincar\'{e} Duality complexes. This 
includes the systematic study of a homotopy theoretic generalization of a gyration, obtained 
from a type of surgery in the manifold case. In particular, for a fixed Poincar\'{e} 
Duality complex, a criterion is given for the possible homotopy types of gyrations 
and shows there are only finitely many. 
\end{abstract}

\maketitle

\section{Introduction} 
Let $M$ be a path-connected Poincar\'{e} Duality complex. For a fixed Poincar\'{e} Duality complex~$T$ of the same dimenision, the connected sum $M\# T$ is called the {\it $T$-stabilization} of $M$. Typical choices of~$T$ are a projective space or a product of spheres. When the spaces $M$ and $T$ are manifolds, the notion of $T$-stabilization was introduced by Kreck \cite{K2}, who suggested studying the $T$-stable classification of manifolds. This became an important and active problem in geometric topology. 
When $T=\mathbb{C}P^2$, the $\mathbb{C}P^2$-stable classification of smooth $4$-manifolds was studied by Kasprowski, Powell and Teichner \cite{KPT} based on \cite{K1}. 
When $T=S^n\times S^n$, the classification of stable diffeomorphism classes of $2n$-manifolds was systematically studied by Kreck~\cite{K1} using his modified surgery technique. 

Homotopy theoretic properties of the $T$-stabilization of a manifold have also been intensively investigated. When $T=\mathbb{C}P^n$ or $\mathbb{H}P^n$, the authors~\cite{HT1} proved loop space decompositions of $\mathbb{C}P^n$- and $\mathbb{H}P^n$-stabilizations by mixing techniques from both geometric and algebraic topology. 
When~$T$ is a product of spheres, Beben and the second author \cite{BT1} gave loop space decompositions of the corresponding $T$-stabilizations, with a further generalization by the second author \cite{T1} to the case when attaching map of the highest dimensional cell of $T$ satisfies an ``inert" condition. Jeffrey and Selick \cite{JS}, Chenery \cite{C} and Basu and Ghosh~\cite{BG} studied a recognition problem for $T$-stabilization with respect to certain fibrations, while the second author \cite{T2} proved a weak version of Moore's conjecture for $T$-stabilization.

In this paper, we generalize the homotopy theoretic study of $T$-stabilizations of manifolds to the more general context of Poincar\'{e} Duality complexes. The key for doing this is the 
development of a purely homotopy theoretical argument in the case of projective stabilizations 
that replaces the geometric argument used in~\cite{HT1}. This validates the argument  
for Poincar\'{e} Duality complexes as well as manifolds. An important intermediary space in the 
manifold case was a gyration, obtained as the result of a certain surgery. In our approach  
this is replaced by a generalization called a homotopy gyration. We systematically study 
its homotopy theory and classify its possible homotopy types. 

We present our results in four parts: homotopy gyrations, projective stabilizations, other stabilizations, and an application to $4$-manifolds.
\medskip

 \noindent 
 \textbf{Homotopy gyrations}.~ 
In \cite{HT1}, the {\it topological gyration} of an $n$-dimensional manifold $M$ with framing $\tau: S^{k-1}\larrow SO(n)$ is defined as the effect of the $k$-surgery with framing $\tau$ on the product manifold $M\times S^{k-1}$ along the canonical embedding of $S^{k-1}$. This construction is crucial for the study of the homotopy theory of projective stabilizations of manifolds \cite{HT1}, toric topology \cite{GLdM} and regular circle actions on manifolds \cite{D}. Special cases of gyration constructions were studied by Gonz\'{a}lez-Acu\~{n}a~\cite{GA} and by Duan \cite{D} from a geometric perspective. 
 
A homotopy theoretic generalization of a gyration for Poincar\'{e} Duality complexes was introduced 
in~\cite{CT}. We establish properties that will be crucial for the study of projective stabilizations in the context of Poincar\'{e} Duality complexes. 
Let~$M$ be a path-connected $n$-dimensional Poincar\'{e} Duality complex with a single $n$-cell, which will be referred to as the top cell. Let $\overline{M}$ be~$M$ with a point deleted, or equivalently up to homotopy, the $(n-1)$-skeleton of $M$. The {\it homotopy gyration} of $M$ with framing $\tau: S^{k-1}\larrow SO(n)$ is defined as the homotopy pushout
\[
  \label{gyrationpointro} 
  \diagram 
      S^{n-1}\times S^{k-1}\dto^-{\pi_1}\rto^{ \widetilde{\tau}}  &  S^{n-1}\times S^{k-1} \rto^{f\times 1}  &
     \overline{M}\times S^{k-1}  \dto \\ 
      S^{n-1}\rrto && \Gtau(M), 
  \enddiagram 
\]
where $\pi_1$ is the projection onto the first factor, $\widetilde{\tau}(a, t)=(\tau(t)a, t)$, and $f$ is the attaching map for the top cell of $M$. When $\tau$ is trivial, or equivalently when $\widetilde{\tau}$ is the identity map, we write $\tau=0$ and denote $\Gtau(M)$ by $\Gzero(M)$. 

The following theorem summarizes our results on the properties of homotopy gyrations. Recall that the classical $J$-homomorphism is a map $J: \pi_{k-1}(SO(n))\larrow \pi_{n+k-1}(S^n)$ and it is stable when $n\geq k+2$. 

\begin{theorem} 
   \label{gyrationtypeintro} 
   Let $M$ be a path-connected $n$-dimensional Poincar\'{e} Duality complex with a single top cell. Let $\tau: S^{k-1}\larrow SO(n)$ be a map with $k\geq 2$.
   
If $\tau$ is trivial, then there is a homotopy equivalence
   \[
   \Omega\Gzero(M)\simeq\Omega \overline{M}\times\Omega\Sigma^k H,
   \]
where $H$ is the homotopy fibre of the top cell attachment $S^{n-1}\larrow \overline{M}$. 
   
In general, for a fixed $M$ and $k$ with $n\geq k+2$, the possible homotopy types of the homotopy gyrations $\Gtau(M)$ are determined by the homotopy class $J([\tau])$, In particular, there are only finitely many possible homotopy types. 

Furthermore, $\Gtau(M)$ satisfies the following: 
   \begin{itemize}
   \item[(1)] if $k\equiv 3, 5, 6, 7 \ {\rm mod}  \ 8$ then $\Gtau(M)\simeq \Gzero(M)$;   
   \item[(2)] if $k\equiv 1, 2\ {\rm mod} \ 8$ then $\Gtau(M)\simeq \Gzero(M)$ after localization away from $2$; 
   \item[(3)] if $k=4s$ then $\Gtau(M)\simeq \Gzero(M)$ after localization away from all primes $p$ such that $p-1~|~2s$.
   \end{itemize}
\end{theorem} 
Theorem \ref{gyrationtypeintro} generalizes and strengthens the corresponding results for topological gyrations in~\cite{HT1}. It is useful for classifying the homotopy type of gyrations in concrete cases with further information specific to $M$. Such a classification was recently given by Chenery and the second 
author \cite{CT} in the context of projective planes. The two cases when $k=2$ and $4$ are special and are crucial for the study of projective stabilizations.

\begin{proposition}[Corollaries \ref{k=2gyrationcoro} and \ref{k=4gyrationcoro}]\label{k=2,4gyrationprop}
Let $M$ be a path-connected $n$-dimensional Poincar\'{e} Duality complex with a single top cell. Let $\tau: S^{k-1}\larrow SO(n)$ be a map. 
\begin{itemize}
\item[(1)] When $k=2$ and $n\geq 4$, the homotopy type of a homotopy gyration of $M$ is classified by $\pi_{n+1}(S^n)\cong \mathbb{Z}/2\mathbb{Z}$, so there are at most two possible homotopy types. 
\item[(2)]
When $k=4$ and $n\geq 6$, the homotopy type of a homotopy gyration of $M$ is classified by $\pi_{n+3}(S^n)\cong \mathbb{Z}/24\mathbb{Z}$, so there are at most $24$ possible homotopy types. 
\end{itemize}
\end{proposition}
\medskip

 \noindent 
 \textbf{Projective stabilizations}.~ 
When $T$ is a complex projective space $\mathbb{C}P^n$ or a quaternionic projective space $\mathbb{H}P^n$, we prove the existence of principal homotopy fibrations relating $T$-stabilizations and homotopy gyrations. 

In the $k=2$ and $n\geq 4$ case, write the two possible homotopy types for 
homotopy gyrations in Proposition~\ref{k=2,4gyrationprop} as $\mathcal{G}^{0}_{\mathbb{C}}(M)$ 
for the trivial case and $\mathcal{G}^{1}_{\mathbb{C}}(M)$ for the nontrivial case. If $n$ is 
a positive integer and $n\equiv t~{\rm mod}~2$ then write $\mathcal{G}^{n}_{\mathbb{C}}$ 
for $\mathcal{G}^{t}_{\mathbb{C}}$. 

In the $k=4$ and $n\geq 6$ case, write the $24$ possible homotopy types for homotopy gyrations in 
Proposition~\ref{k=2,4gyrationprop} as $\mathcal{G}^{0}_{\mathbb{H}}(M)$ for the trivial 
case and $\mathcal{G}^{t}_{\mathbb{H}}(M)$ for $1\leq t\leq 23$ in the nontrivial cases. 
There is an ambiguous sign, explained in more detail before Theorem~\ref{fibhp}, 
that leads us to define~$\overline{n}$ as an appropriate sign times $n$. If $\overline{n}$ 
is an even integer and $\overline{n}\equiv 2t~{\rm mod}~48$ then write 
$\mathcal{G}^{\frac{\overline{n}}{2}}_{\mathbb{H}}$ for $\mathcal{G}^{t}_{\mathbb{H}}$.

\begin{theorem}[Theorems \ref{fibcp} and \ref{fibhp}] 
   \label{fibfpintro} 
   Let $M$ be a path-connected $2n$-dimensional Poincar\'{e} Duality complex with a single top cell and $n\geq 2$. 
   \begin{itemize}
   \item[(1)] 
   There is a principal homotopy fibration 
   \[\nameddright{S^{1}}{j}{\mathcal{G}^n_{\mathbb{C}}(M)}{}{M\conn\cp^{n}}\] 
   where the map~$j$ is null homotopic.  
   \item[(2)] 
   If $n$ is even and $n\geq 4$ then there is a principal homotopy fibration 
   \[\nameddright{S^{3}}{j}{\mathcal{G}^{\frac{\overline{n}}{2}}_{\mathbb{H}}(M)}{}{M\conn\hp^{\frac{n}{2}}}\] 
   where the map $j$ is null homotopic.
   \end{itemize}
\end{theorem} 

Theorem \ref{fibfpintro} generalizes the corresponding results of \cite[Section 3]{HT1} and \cite{D} from manifolds to Poincar\'{e} Duality complexes, while also explicitly identifying the framings of the gyrations involved. Critical to this is the development of a purely homotopy theoretical argument in this paper that replaces the methods from geometric topology used in \cite[Section 3]{HT1} and \cite{D}. 

Notice that the loop space decompositions in Theorem~\ref{gyrationtypeintro} 
depend only on $k$ and are indepedent of $n$, so the sign ambiguity in Theorem~\ref{fibfpintro}~(b) 
has no effect on the loop space decomposition. 
Combining the two theorems immediately gives the following loop space decompositions of projective stabilizations, that again generalize and strengthen \cite[Theorem 1.2]{HT1} from manifolds to Poincar\'{e} Duality complexes. 
\begin{theorem}\label{projstabdecthmintro}
Let $M$ be a path-connected $2n$-dimensional Poincar\'{e} Duality complex with a single top cell and $n\geq 2$. Let $H$ be the homotopy fibre of the top cell attachment $S^{2n-1}\larrow \overline{M}$.
   \begin{itemize}
   \item[(1)] If $n$ is even then there is a homotopy equivalence
   \[
   \Omega (M\conn\cp^{n}) \simeq S^1\times \Omega \overline{M}\times\Omega\Sigma^2 H.
   \]
   \item[(2)] If $n$ is odd then there is a homotopy equivalence after localization away from $2$
   \[
   \Omega (M\conn\cp^{n}) \simeq S^1\times \Omega \overline{M}\times\Omega\Sigma^2 H.
   \]
   \item[(3)] If $n\equiv 0~{\rm mod}~48$ then there is a homotopy equivalence
   \[
   \Omega (M\conn\hp^{\frac{n}{2}}) \simeq S^3\times \Omega \overline{M}\times\Omega\Sigma^4 H.
   \]
  \item[(4)] If $n$ is even and $n\equiv 6x\bmod{48}$ for some integer $x$ then there is a homotopy equivalence after localization away from $2$
   \[
   \Omega (M\conn\hp^{\frac{n}{2}}) \simeq S^3\times \Omega \overline{M}\times\Omega\Sigma^4 H.
   \]
   \item[(5)] If $n$ is even and $n\equiv 16\bmod{48}$ or $n\equiv 32\bmod{48}$ then there is a homotopy equivalence after localization away from $3$
   \[
   \Omega (M\conn\hp^{\frac{n}{2}}) \simeq S^3\times \Omega \overline{M}\times\Omega\Sigma^4 H.
   \]
   \item[(6)] If $n\geq 4$ is even and $n\equiv x\bmod{48}$ for $x\in 
   \{2, 4, 8, 10, 14, 20, 22, 26, 28, 34, 38, 40, 44, 46 \}$ then there is a homotopy equivalence after localization away from $2$ and $3$
   \[\hspace{5cm}
   \Omega (M\conn\hp^{\frac{n}{2}}) \simeq S^3\times \Omega \overline{M}\times\Omega\Sigma^4 H.
   \hspace{5cm}\qqed\]
   \end{itemize}
\end{theorem} 

 \noindent 
 \textbf{Other stabilizations}.~ 
In general, for path-connected spaces $X$ and $Y$, let $X\rtimes Y$ be the quotient space obtained 
from $X\times Y$ by collapsing the subspace $\ast\times Y$ to a point. 

\begin{theorem}[Theorem \ref{fibid}]
   \label{fibidintro} 
   Let $M$ be a path-connected $n$-dimensional Poincar\'{e} Duality complex with a single top cell. Let $T$ 
   be an $(m-1)$-connected $n$-dimensional Poincar\'{e} Duality complex with $2\leq m<n$. If there is a map $T\stackrel{h}{\larrow} S^m$ that has a right homotopy inverse then there is a homotopy fibration 
   \[\nameddright{E\vee (\overline{M}\rtimes\Omega S^{m})}{}{M\conn T}{h'}{S^{m}},\] 
   where $E$ is the homotopy fibre of $h$, and $h'$ has a right homotopy inverse. 
\end{theorem} 

An interesting case of Theorem~\ref{fibidintro} is when $T$ is a product of spheres.
\begin{theorem}[Theorem \ref{sphereprodstabilize}] 
   \label{sphereprodstabilizeintro} 
   Let $M$ be a path-connected $n$-dimensional Poincar\'{e} Duality complex with a single top cell.    
   If $2\leq m\leq n-m$ then there is a homotopy fibration 
   \[\nameddright{S^{n-m}\vee(\overline{M}\rtimes\Omega S^{m})}{}{M\conn(S^{m}\times S^{n-m})}{h'}{S^{m}}\] 
   where $h'$ has a right homotopy inverse.  
\end{theorem} 

In \cite[Theorem 1.4]{T1} a similar result was proved: if $M$ is simply-connected then there is 
a homotopy fibration 
\(\nameddright{\overline{M}\rtimes(\Omega S^{m}\times\Omega S^{n-m})}{}{M\conn(S^{m}\times S^{n-m})} 
     {h''}{S^{m}\times S^{n-m}}\) 
where $\Omega h''$ has a right homotopy inverse. This was generalized in~\cite[Theorem 10.6]{H} to 
the path-connected case. The advantage of the formulation in Theorem~\ref{sphereprodstabilizeintro} is 
that the homotopy fibre, because of the wedge summand, is more 
accessible to a finer analysis using the Hilton-Milnor or Ganea Theorems. 
\medskip 

 \noindent 
 \textbf{An application}.~ 
A particularly interesting case of Theorems~\ref{fibfpintro} and \ref{sphereprodstabilizeintro} is for $4$-manifolds. 

\begin{corollary} 
   \label{4manifold} 
   Let $M$ be a path-connected closed orientable $4$-manifold. Then there is a principal homotopy fibration 
      \[\nameddright{S^1}{j}{\mathcal{G}^{0}(M)}{}{M\conn\cp^{2}}\] 
   where $j$ is null homotopic, and a homotopy fibration  
   \[\nameddright{S^{2}\vee(\overline{M}\rtimes\Omega S^{2})}{}{M\conn(S^{2}\times S^{2})}{h'}{S^{2}}\] 
   where $h'$ has a right homotopy inverse.
~$\qqed$ 
\end{corollary} 

Corollary~\ref{4manifold} implies a rigidity result in the case of $4$-manifolds with free fundamental group. 

\begin{proposition} 
   \label{rigidity} 
   Let $M$ and $N$ be path-connected closed orientable $4$-manifolds with free fundamental group. If 
   there is a module isomorphism $H_{\ast}(\overline{M};\mathbb{Z})\cong H_{\ast}(\overline{N};\mathbb{Z})$ 
   then $\pi_{\ast}(M\conn(S^{2}\times S^{2}))\cong\pi_{\ast}(N\conn(S^{2}\times S^{2}))$. 
\end{proposition} 

\begin{proof} 
By~\cite{KM}, a path-connected closed orientable $4$-manifold with free fundamental group has its 
$3$-skeleton homotopy equivalent to a wedge of $1$, $2$ and $3$-dimensional spheres. Thus  
$\overline{M}$ and $\overline{N}$ are both homotopy equivalent to wedges of spheres. 
The module isomorphism $H_{\ast}(\overline{M};\mathbb{Z})\cong H_{\ast}(\overline{N};\mathbb{Z})$ 
therefore implies that $\overline{M}$ and $\overline{N}$ have the same number of spheres in each dimension, 
and so are homotopy equivalent. It follows that there is a homotopy equivalence 
$\overline{M}\rtimes\Omega S^{2}\simeq\overline{N}\rtimes\Omega S^{2}$, from which Corollary~\ref{4manifold} 
implies that there is a homotopy equivalence 
$\Omega(M\conn(S^{2}\times S^{2}))\simeq\Omega (N\conn(S^{2}\times S^{2}))$. 
As the based loop space shifts homotopy groups down one dimension, there is an isomorphism 
$\pi_{\ast}(M\conn(S^{2}\times S^{2}))\cong\pi_{\ast}(N\conn(S^{2}\times S^{2}))$. 
\end{proof} 

It is not clear that a result similar to Proposition \ref{rigidity} holds for $\cp^2$-stabilization. By Theorem~\ref{projstabdecthmintro}, if $M$ is a path-connected closed orientable $4$-manifold then 
\[
\Omega (M\# \cp^2)\cong S^1\times \Omega \overline{M}\times \Omega \Sigma^2 H,
\]
where $H$ is the homotopy fibre of the top cell attachment $S^3\larrow \overline{M}$. The homotopy type of $H$ could be sufficiently wild to affect the homotopy groups of $M\# \cp^2$ dramatically.
\medskip

 \noindent 
 \textbf{Organization of the paper}.~ 
 In Section \ref{sec:background} we review the Cube Lemma and two decomposition theorems that will be used. 
 In Section \ref{sec:gyration} we introduce a homotopy gyration and study its properties.
 In Section~\ref{sec:gyrationhgp} we prove a loop space decomposition for a homotopy gyration with trivial framing and prove Theorem \ref{gyrationtypeintro}.  
In Section~\ref{sec:cpn} we study projective stabilizations and prove Theorem \ref{fibfpintro}.
Sections~\ref{sec:Tstab} and~\ref{sec:Tstab2} are devoted to other stabilizations, where two proofs of Theorem \ref{fibidintro} are provided. 
In Section~\ref{sec:spherestab} we specialize to stabilizations by a product of spheres and give several examples.  

$\, $

\noindent{\bf Acknowledgements.} The first author was supported in part by the National Natural Science Foundation of China (Grant nos. 12331003 and 12288201), the National Key R\&D Program of China (No. 2021YFA1002300) and the Youth Innovation Promotion Association of Chinese Academy Sciences. 
The second author would like to thank the Chinese Academy of Sciences for 
its generous hospitality in supporting a research visit during which this paper was completed.

\section{Background} 
\label{sec:background} 
In this section, we review three useful tools in unstable homotopy theory.
\medskip

 \noindent 
 \textbf{The Cube Lemma}. 
We state a version of Mather's Cube Lemma~\cite{Ma}. 

\begin{lemma}
   \label{cube}
   Suppose that there is a homotopy pushout 
   \[\diagram 
          A\rto\dto & B\dto \\ 
          C\rto & D   
      \enddiagram\] 
   and a map 
   \(h\colon\namedright{D}{}{Z}\). 
   For $X$ one of $A,B,C$ or $D$, let $F_{X}$ be the homotopy fibre of the composite 
   \(\nameddright{X}{}{D}{h}{Z}\). Then there is a homotopy commutative cube 
   \[\spreaddiagramcolumns{-1pc}\spreaddiagramrows{-1pc}
   \diagram
        F_{A}\rrto\drto\ddto & & F_{B}\dline\drto & \\
        & F_{C}\rrto\ddto & \dto & F_{D}\ddto \\
        A\rline\drto & \rto & B\drto & \\
        & C\rrto & & D
   \enddiagram\] 
   where the four sides are homotopy pullbacks and the top face is a homotopy pushout. 
~$\qqed$
\end{lemma} 

Technically, Mather proved a more general result without the hypothesis that the vertical maps 
are obtained by taking fibres over a common base space, for which he needed a different definition of a 
``homotopy commutative cube". In our case the stronger hypothesis lets one use~\cite[Lemma~3.1]{PT}, 
for example, to establish Lemma~\ref{cube}, or one can refer to the monograph \cite[Section 5.1]{HT3}. 
\medskip

\noindent 
\textbf{A decomposition theorem}. 
We state a result from~\cite{BT2}. Suppose that there is a homotopy cofibration 
\(\nameddright{A}{f}{X}{}{X'}\) 
and a homotopy fibration 
\(\nameddright{E}{}{X}{h}{Z}\). 
Suppose that $h\circ f$ is null homotopic so that $h$ extends across 
\(\namedright{X}{}{X'}\) 
to a map 
\(h'\colon\namedright{X'}{}{Z}\). 
Let $E'$ be the homotopy fibre of $h'$. This data is arranged in a diagram 
\begin{equation} 
  \label{data} 
  \diagram 
       & E\rto\dto & E'\dto \\ 
       A\rto^-{f} & X\rto\dto^{h} & X'\dto^{h'} \\ 
       & Z\rdouble & Z 
  \enddiagram 
\end{equation} 
where the two columns form a homotopy fibration diagram. 

As noted in the Introduction, for pointed spaces $A$ and $B$, 
the \emph{right half-smash} is defined as the quotient space 
\[A\rtimes B=(A\times B)/\sim\] 
where $(\ast,b)\sim(\ast,\ast)$. It is well known that if $A$ is a co-$H$-space then there is a 
homotopy equivalence $A\rtimes B\simeq A\vee (A\wedge B)$.  

\begin{theorem} 
   \label{BT} 
   Given a diagram of data~(\ref{data}). If $\Omega h$ has a right homotopy inverse 
   \(s\colon\namedright{\Omega Z}{}{\Omega X}\) 
   then there is a homotopy cofibration 
   \[\nameddright{A\rtimes\Omega Z}{\Gamma}{E}{}{E'}\] 
   for some map $\Gamma$ whose restriction to $A$ is a lift of $f$.~$\qqed$ 
 \end{theorem} 
 
 \begin{remark}\label{BTremark}
 Following \cite{BT2} or \cite[Section 5.4]{HT3}, the map $\Gamma$ can be defined as follows. The homotopy cofibration 
 \(\nameddright{A}{f}{X}{}{X'}\) 
 implies that the map $f$ lifts to the homotopy fibre $F$ of 
 \(\namedright{X}{}{X'}\). 
 The homotopy pullback in the upper square of~(\ref{data}) implies that $F$ is also the homotopy 
 fibre of the map 
 \(\namedright{E}{}{E'}\). 
 The factorization 
 \(\nameddright{A}{}{F}{}{E}\) 
 gives a choice of lift $\mathfrak{f}: A\larrow E$ through $E\larrow X$. Further, the homotopy fibration $E\larrow X \stackrel{h}{\larrow}Z$ extends to a principal homotopy fibration $\Omega Z\stackrel{\partial}{\larrow} E\stackrel{}{\larrow} X$ such that the connecting map $\partial$ is null homotopic. Then there is a homotopy action $\vartheta: E\times \Omega Z\longrightarrow E$, whose restriction on $\Omega Z$ is null homotopic. It follows that $\vartheta$ reduces to a map $\theta: E\rtimes \Omega Z \stackrel{}{\larrow} E$. For a suitable choice of the reduced action $\theta$, the map $\Gamma$ is be defined as the composite
\[
\Gamma: A\rtimes\Omega Z \stackrel{\mathfrak{f}\rtimes 1}{\larrow} E\rtimes\Omega Z \stackrel{\theta}{\larrow} E.
\]
  \end{remark}

 \begin{remark}\label{BTnatremark}
 Theorem \ref{BT} has a naturality property as stated in \cite[Remark 2.7]{T1}. If there is a homotopy fibration diagram 
\[\diagram 
      E\rto^-{}\dto^{} & X\rto^-{h}\dto^{} & Z\dto^{}  \\ 
      \widehat{E}\rto^-{} &\widehat{X}\rto^-{\widehat{h}} & \widehat{Z}
  \enddiagram\] 
and both $\Omega h$ and $\Omega \widehat{h}$ have right homotopy inverses $s$ and $s'$, 
respectively, such that there is a homotopy commutative diagram 
\begin{equation}\label{ss'diag}
\diagram 
      \Omega {Z}\rto^-{s}\dto^{} & \Omega X\dto^{} \\ 
      \Omega \widehat{Z}\rto^-{s'} & \Omega \widehat{X}, 
  \enddiagram
  \end{equation}
then there is a homotopy cofibration diagram
\[\diagram 
      A\rtimes \Omega Z \rto^-{\Gamma} \dto^{} & E\rto^-{}\dto^{} & E'\dto^{}  \\ 
      \widehat{A}\rtimes \Omega \widehat{Z}\rto^-{\widehat{\Gamma}}  & \widehat{E}\rto^-{} & \widehat{E}'. 
  \enddiagram\] 
 \end{remark}

 The point of Theorem~\ref{BT} is that the map $\Gamma$ can sometimes be used to 
 determine the homotopy type of $E'$. The right homotopy inverse for $\Omega h$ 
 implies there is a right homotopy inverse for $\Omega h'$, resulting in a homotopy 
 equivalence $\Omega X'\simeq\Omega Z\times\Omega E'$. Knowing the homotopy 
 type of $E'$ then informs on the homotopy type of $\Omega X'$. 
 \medskip

\noindent 
\textbf{Another decomposition}. 
We set up and state a result from~\cite{T3}. 
Given data as in~(\ref{data}) having the property that $\Omega h$ has a right homotopy 
inverse, Theorem~\ref{BT} states that there is a homotopy cofibration 
\[\nameddright{A\rtimes\Omega Z}{\Gamma}{E}{}{E'}.\]
In this section a criterion is proved that implies $\Gamma$ has a left homotopy inverse under certain 
hypotheses and there is a homotopy equivalence $E\simeq (A\rtimes\Omega Z)\vee E'$. 

Suppose that there is a map 
\(\delta\colon\namedright{X}{}{\Sigma Y}\)  
and a homotopy co-action 
\[\psi\colon\namedright{X}{}{X\vee\Sigma Y}\] 
with respect to $\delta$. Suppose as well that there is a diagram of data 
\begin{equation} 
  \label{criteriondata} 
  \diagram 
       & E\rto\dto & E'\dto  \\ 
       \Sigma Z\wedge Y\rto^-{f} & X\rto\dto^{h} & X'\dto^{h'} \\ 
       & \Sigma Z\rdouble & \Sigma Z 
  \enddiagram 
\end{equation} 
where the middle row is a homotopy cofibration and the two columns form 
a homotopy fibration diagram. Let $\gamma$ be the composite 
\begin{equation} 
  \label{gammadef} 
  \gamma\colon\nameddright{X}{\psi}{X\vee\Sigma Y}{h\vee 1}{\Sigma Z\vee\Sigma Y}.  
\end{equation} 
Since $\psi$ is a co-action, the composite $p_{1}\circ\gamma$ is homotopic to $h$.  
Define the space $D$ by the homotopy cofibration diagram 
\begin{equation} 
  \label{Ddef} 
  \diagram 
        \Sigma Z\wedge Y\rto^-{f}\ddouble & X\rto\dto^{\gamma} & X'\dto \\ 
        \Sigma Z\wedge Y\rto^-{\gamma\circ f} & \Sigma Z\vee\Sigma Y\rto &  D. 
  \enddiagram 
\end{equation}  

The prototype to think of is  when $\gamma\circ f$ is homotopic to the Whitehead 
product of the inclusions of $\Sigma Z$ and $\Sigma Y$ into $\Sigma Z\vee\Sigma Y$,  
in which case $D\simeq\Sigma Z\times\Sigma Y$. However, we wish to allow for more  
flexibility in terms of the homotopy class of $\gamma\circ f$. To get this, observe that as 
$p_{1}\circ\gamma\simeq h$ and $h\circ f$ is null homotopic, there is a null  homotopy 
for $p_{1}\circ\gamma\circ f$, implying that the pinch map  
\(\namedright{\Sigma Z\vee\Sigma Y}{p_{1}}{\Sigma Z}\) 
extends to a map 
\(\namedright{D}{\mathfrak{h}}{\Sigma Z}\). 
Let $g$ be the composite 
\[g\colon\nameddright{\Sigma Y}{i_{2}}{\Sigma Z\vee\Sigma Y}{}{D}.\] 
Then in place of $\gamma\circ f$ being a Whitehead product, giving $D\simeq\Sigma Z\times\Sigma Y$ 
and a homotopy fibration 
\(\nameddright{\Sigma Y}{g}{D}{\mathfrak{h}}{\Sigma Z}\), 
we will only assume the existence of the homotopy fibration. 

\begin{theorem}  
   \label{splittingprinciple} 
   Suppose that there is a map 
   \(\delta\colon\namedright{X}{}{\Sigma Y}\)  
   and a homotopy co-action 
   \(\psi\colon\namedright{X}{}{X\vee\Sigma Y}\) 
   with respect to $\delta$, and data as in~(\ref{criteriondata}). If  
   \begin{letterlist} 
      \item $\Omega h$ has a right homotopy inverse and 
      \item there is a homotopy fibration 
               \(\nameddright{\Sigma Y}{g}{D}{\mathfrak{h}}{\Sigma Z}\), 
   \end{letterlist}    
   then the homotopy cofibration 
   \(\nameddright{(\Sigma Z\wedge Y)\rtimes\Omega\Sigma Z}{\Gamma}{E}{}{E'}\)  
   obtained by applying Theorem~\ref{BT} to~(\ref{criteriondata}) splits: the map $\Gamma$ has 
   a left homotopy inverse and there is a homotopy equivalence 
   \[E\simeq((\Sigma Z\wedge Y)\rtimes\Omega\Sigma Z)\vee  E'.\] 
\end{theorem} 
\vspace{-0.9cm}~$\qqed$\medskip

\section{Homotopy gyrations} 
\label{sec:gyration} 
In this section, we study the homotopy generalization of a topological gyration and its properties. Two special cases are discussed in detail as they will be used in the sequel. 

Let $M$ be a path-connected $n$-dimensional Poincar\'{e} duality complex with a single top cell. Let $\overline{M}$ be $M$ with a point deleted, or equivalently, the $(n-1)$-skeleton of $M$ up to homotopy equivalence. There is a homotopy cofibration
\[
S^{n-1}\stackrel{f}{\larrow}\overline{M}\larrow M 
\] 
where $f$ attaches the top cell to $M$. 
Let $\tau: S^{k-1}\larrow SO(n)$ be a based map with $k\geq 2$. Using the standard action of $SO(n)$ on $S^{n-1}$, define the map 
\[
\widetilde{\tau}\colon\namedright{S^{n-1}\times S^{k-1}}{}{S^{n-1}\times S^{k-1}} 
\]
by $\widetilde{\tau}(a, t)=(\tau(t)a, t)$. 

\begin{lemma}\label{taulemma}
The map $\widetilde{\tau}$ is a homeomorphism, and restricts to the identity map on the factor $S^{n-1}$.
\end{lemma}
\begin{proof}
Define $\tau^{-1}: S^{k-1}\larrow SO(n)$ by $\tau^{-1}(t)=(\tau(t))^{-1}$. It follows that $\tau^{-1}\cdot \tau$ and $\tau\cdot \tau^{-1}$ are constant maps onto the basepoint, which is the identity matrix. Therefore $\widetilde{\tau}\circ \widetilde{\tau^{-1}}=\widetilde{\tau^{-1}}\circ \widetilde{\tau}=1$ and $\widetilde{\tau}$ is a homeomorphism. 

Restricting the map $\widetilde{\tau}$ to the first factor $S^{n-1}=S^{n-1}\times \{\ast\}$, we see that $\widetilde{\tau}(a, \ast)=(\tau(\ast) a, \ast)=(a,\ast)$, that is, $\widetilde{\tau}$ restricts to the identity map on $S^{n-1}$.
\end{proof}

A {\it homotopy gyration} $\mathcal{G}^{\tau}(M)$ determined by the {\it framing} $\tau$ is defined by the homotopy pushout 
\begin{equation} 
  \label{gyrationpo} 
  \diagram 
      S^{n-1}\times S^{k-1}\dto^-{\pi_1}\rrto^{(f\times 1)\circ \widetilde{\tau}} & &\overline{M}\times S^{k-1}  \dto \\ 
      S^{n-1}\rrto & & \Gtau(M), 
  \enddiagram 
\end{equation} 
where $\pi_1$ is the projection onto the first factor. When $\tau$ is trivial, or equivalently when $\widetilde{\tau}$ is the identity map, write $\tau=0$ and denote $\Gtau(M)$ by $\Gzero(M)$. 
There have been several stages in the development of a gyration. If $M$ 
is a manifold, $k=2$ and $\tau$ is trivial, this was first defined by Gonz\'{a}lez Acu\~{n}a~\cite{GA}; 
it was later generalized to a manifold, $k=2$ and any framing $\tau$ by Duan~\cite{D}; then to a 
manifold, any $k$ and any framing by the authors~\cite[Section 2]{HT1}; then to a 
simply-connected Poincar\'{e} Duality complex, any $k$ and any framing in~\cite[Section 3]{CT}; 
and finally to any path-connected Poincar\'{e} Duality complex with a single top cell, any $k$ 
and any framing as above. 

In general, for pointed, path-connected spaces $A$ and $B$, let 
\[q\colon\namedright{A\times B}{}{A\rtimes B}\] 
be the quotient map to the half-smash. The projection 
\[\pi_{1}\colon\namedright{A\times B}{}{A}\] 
to the first factor extends to a canonical projection 
\[\overline{\pi}_{1}\colon\namedright{A\rtimes B}{}{A}\] 
defined by $\overline{\pi}_{1}(a,b)=a$. Note that $\pi_{1}=\overline{\pi}_{1}\circ q$. 

\begin{lemma} 
  \label{taureduce} 
If $3\leq k+1\leq n$ then there is a homotopy commutative diagram 
  \[\diagram 
        S^{n-1}\times S^{k-1}\rto^-{\widetilde{\tau}}\dto^{q} & S^{n-1}\times S^{k-1}\dto^{q} \\ 
        S^{n-1}\rtimes S^{k-1}\rto^-{\overline{\tau}} & S^{n-1}\rtimes S^{k-1} 
    \enddiagram\] 
  for some map $\overline{\tau}$. 
\end{lemma}

\begin{proof} 
Recall there is a standard fibration $SO(n-1)\larrow SO(n)\stackrel{{\rm ev}}{\larrow} S^{n-1}$, where ${\rm ev}$ evaluates a rotation of $S^{n-1}$ at the basepoint $\ast$. Since $\pi_{k-1}(S^{n-1})=0$ when $n\geq k+1$, the composite ${\rm ev}\circ \tau$ is null homotopic, implying that there is a lift
\[
\diagram 
      & S^{k-1}\dto^{\tau}\dlto_{} & \\ 
      SO(n-1)\rto & SO(n)\rto^-{{\rm ev}} & S^{n-1}.
  \enddiagram\] 
In other words, the image of $\tau$ fixes the basepoint of $S^{n-1}$ up to homotopy. Since our focus is on the homotopy type of $\Gtau(M)$, we may safely assume that $\tau(t)(\ast) = \ast$ for all $t \in S^{k-1}$. Then $\widetilde{\tau}(\ast, t)=(\ast, t)$, that is, $\widetilde{\tau}$ restricts to the identity map on the second factor $S^{k-1}$. It follows that the copy of $S^{k-1}$ on both sides of the map 
\(\namedright{S^{n-1}\times S^{k-1}}{\widetilde{\tau}}{S^{n-1}\times S^{k-1}}\) 
may be simultaneously collapsed out so that $\widetilde{\tau}$ reduces to a map 
\[
\overline{\tau}: S^{n-1}\rtimes S^{k-1} \larrow S^{n-1}\rtimes S^{k-1} 
\]
that makes the diagram in the statement of the lemma homotopy commute. 
\end{proof} 

\begin{lemma}\label{gyrationpolemma}
If $3\leq k+1\leq n$ then there is a homotopy pushout
\begin{equation} 
  \label{gyrationpo2} 
  \diagram 
      S^{n-1}\rtimes S^{k-1}\dto^-{\overline{\pi}_1}\rrto^{(f\rtimes 1)\circ \overline{\tau}} 
         & &\overline{M}\rtimes S^{k-1}  \dto \\ 
      S^{n-1}\rrto & & \Gtau(M).
  \enddiagram 
\end{equation} 
\end{lemma}

\begin{proof}
Consider the diagram
\[
\diagram 
      S^{n-1}\times S^{k-1} \rto^{\widetilde{\tau}} \dto^{q}  &  S^{n-1}\times S^{k-1} \dto^{q} \rto^{f\times 1} & \overline{M}\times S^{k-1} \dto^{q} \\
      S^{n-1}\rtimes S^{k-1} \rto^{\overline{\tau}}               & S^{n-1}\rtimes S^{k-1} \rto^{f\rtimes 1} &\overline{M}\rtimes S^{k-1}.  
  \enddiagram 
\] 
The left square homotopy commutes by Lemma~\ref{taureduce} and the right square commutes by the naturality of the half-smash. Since both are induced 
by collapsing out a common copy of $S^{k-1}$, they are also both homotopy pushouts. Now 
consider the diagram 
\[
\diagram
    S^{n-1}\times S^{k-1} \dto^{(f\times 1)\circ \widetilde{\tau}} \rto^{q}  &  S^{n-1}\rtimes S^{k-1} \dto^{(f\rtimes 1)\circ\overline{\tau}} \rto^<<<{\overline{\pi}_1}  & S^{n-1} \dto \\
   \overline{M}\times S^{k-1} \rto^{q}  & \overline{M}\rtimes S^{k-1} \rto^{}  & Z.   
\enddiagram
\] 
The left square is the outer rectangle of the previous diagram and so is a homotopy 
pushout. The right square is a homotopy pushout that defines the space $Z$. The outer rectangle implies that $Z$ is the homotopy out of $(f\times 1)\circ\widetilde{\tau}$ and $\overline{\pi}_1\circ q=\pi_1$, which is homotopy equivalent to $\Gtau(M)$ by~\eqref{gyrationpo}. Thus the right square implies that $\Gtau(M)$ is the homotopy pushout of $(f\rtimes 1)\circ\overline{\tau}$ and $\overline{\pi}_{1}$, proving the lemma.  
\end{proof}

The map $\overline{\tau}$ in Lemma \ref{gyrationpolemma} can be better identified. To this end, we first fix a homotopy equivalence between $S^{n-1}\rtimes S^{k-1}$ and $S^{n-1}\vee S^{n+k-2}$ and then recall the classical $J$-homomorphism. 

In general, the standard quotient map 
\(\namedright{A\times B}{}{A\wedge B}\) 
has the property that its suspension has a right homotopy inverse 
\(\mathfrak{t}\colon\namedright{\Sigma A\wedge B}{}{\Sigma(A\times B)}\) 
which can be chosen so that $\Sigma\pi_{1}\circ\mathfrak{t}$ and $\Sigma\pi_{2}\circ\mathfrak{t}$ 
are null homotopic, where $\pi_{1}$ and $\pi_{2}$ are the projections onto the first and 
second factors respectively. Define the map $j$ by the composite 
\[j\colon\nameddright{S^{n+k-1}}{\mathfrak{t}}{\Sigma(S^{n-1}\times S^{k-1})} 
    {\Sigma q}{\Sigma(S^{n-1}\rtimes S^{k-1})}.\]  
Notice that as $\mathfrak{t}$ is a right homotopy inverse to the suspension of the quotient map 
\(\namedright{S^{n-1}\times S^{k-1}}{}{S^{n-1}\wedge S^{k-1}}\) 
and this quotient map factors through the half-smash, then $j$ is a right homotopy inverse 
to the suspension of the quotient map 
\(\namedright{S^{n-1}\rtimes S^{k-1}}{}{S^{n-1}\wedge S^{k-1}}\). 

\begin{lemma} 
   \label{jsusp} 
   If $k+2\leq n$ then $j\simeq\Sigma\overline{j}$ for a map 
   \(\overline{j}\colon\namedright{S^{n+k-2}}{}{S^{n-2}\rtimes S^{k-1}}\). 
\end{lemma} 

\begin{proof} 
This follows since the hypothesis that $k+2\leq n$ implies that $j$ is in the stable range 
and is the suspension image of a map $\overline{j}$ that is also in the stable range. 
\end{proof} 

Note that $j$ being a right homotopy inverse for the suspension of the quotient map 
\(\namedright{S^{n-2}\rtimes S^{k-1}}{}{S^{n-1}\wedge S^{k-1}}\) 
implies that $\overline{j}$ is a right homotopy inverse for the quotient map itself. In 
particular, $\overline{j}$ induces an isomorphism on degree $n+k-2$ homology. Let 
\[i\colon\namedright{S^{n-1}}{}{S^{n-1}\rtimes S^{k-1}}\]  
be the inclusion and define 
\[e\colon\namedright{S^{n-1}\vee S^{n+k-2}}{}{S^{n-1}\rtimes S^{k-1}}\]  
as the wedge sum of $i$ and $\overline{j}$. Then $e$ induces an isomorphism in homology 
and so is a homotopy equivalence by Whitehead's Theorem.  Let 
\[j_{1}\colon\namedright{S^{n-1}}{}{S^{n-1}\vee S^{n+k-2},}\qquad 
    j_2\colon\namedright{S^{n+k-2}}{}{S^{n-1}\vee S^{n+k-2}}\]   
be the inclusions of the left and right wedge summands respectively and let 
\[p_{1}\colon\namedright{S^{n-1}\vee S^{n+k-2}}{}{S^{n-1},}\qquad 
    p_{2}\colon\namedright{S^{n-1}\vee S^{n+k-2}}{}{S^{n+k-2}}\]
be the pinch maps to the left and right wedge summands respectively.

\begin{lemma} 
   \label{eprops} 
   The homotopy equivalence 
   \(\namedright{S^{n-1}\vee S^{n+k-2}}{e}{S^{n-1}\rtimes S^{k-1}}\) 
   satisfies: 
   \begin{itemize} 
      \item[(a)] $e\circ j_{1}\simeq i$ and $e\circ j_{2}=\overline{j}$;  
      \item[(b)] if $k+2\leq n$ then $\overline{\pi}_1\circ e\simeq p_{1}$. 
   \end{itemize} 
\end{lemma} 

\begin{proof} 
Part~(a) is immediate from the definiition of $e$. For part~(b), the hypothesis that $k+2\leq n$ 
implies that we are in the stable range so it is equivalent to show that 
$\Sigma\overline{\pi}_{1}\circ\Sigma e\simeq\Sigma p_{1}$. Consider the composite 
\(\nameddright{S^{n}\vee S^{n+k-1}}{\Sigma e}{S^{n}\rtimes S^{k-1}}{\Sigma \overline{\pi}_{1}}{S^{n}}\). 
A map out of a wedge is determined by its restrictions to the wedge summands. Observe that 
$\Sigma\overline{\pi}_{1}\circ\Sigma e$ restricted to $S^{n}$ is the identity map. Using the facts 
that $\overline{\pi}_{1}\circ q=\pi_{1}$ by definition of $\overline{\pi}_{1}$ and 
$\Sigma\overline{j}\simeq j$ by Lemma~\ref{jsusp}, the 
restriction of $\Sigma\overline{\pi}_{1}\circ\Sigma e$ to $S^{n+k-1}$ is 
$\Sigma\overline{\pi}_{1}\circ\Sigma\overline{j}\simeq\Sigma\overline{\pi}_{1}\circ j=\Sigma\overline{\pi}_{1}\circ\Sigma q\circ\mathfrak{t}=\Sigma\pi_{1}\circ\mathfrak{t}\simeq\ast$, 
where the null homotopy at the end is due to the choice of $\mathfrak{t}$. Thus 
$\Sigma\overline{\pi}_{1}\circ\Sigma e\simeq\Sigma p_{1}$, as required. 
\end{proof}

The $J$-homomorphism  
\begin{equation} 
  \label{Jhomdef} 
    J: \pi_{k-1} (SO(n))\larrow \pi_{n+k-1}(S^{n})    
\end{equation} 
is defined as follows. Represent an element in $\pi_{k-1}(SO(n))$ by a map 
\(\tau\colon\namedright{S^{k-1}}{}{SO(n)}\).  
Using the standard action 
\(\vartheta\colon\namedright{S^{n-1}\times SO(n)}{}{S^{n-1}}\) 
we obtain a composite 
\begin{equation} 
  \label{Jhom} 
  \nameddright{S^{n-1}\times S^{k-1}}{1\times \tau}{S^{n-1}\times SO(n)}{\vartheta}{S^{n-1}}. 
\end{equation}  
Suspending and precomposing~(\ref{Jhom}) with 
\(\namedright{S^{n_k-1}}{\mathfrak{t}}{\Sigma(S^{n-1}\times S^{k-1})}\) 
gives a map 
\(\tau'\colon\namedright{S^{n+k-1}}{}{S^{n}}\). 
At the level of maps, define $J(\tau)=\tau'$. It is a standard fact that~$J$ is a homotopy 
invariant so it induces the map on homotopy groups in~(\ref{Jhomdef}), and this is 
a natural group homomorphism. 

It is also classical that the $J$-homomorphism satisfies a stability property. Let 
\(i\colon\namedright{SO(n-1)}{}{SO(n)}\) 
be the standard inclusion and let 
\(E\colon\namedright{\pi_{m}(S^{t})}{}{\pi_{m+1}(S^{t+1})}\) 
be the suspension map sending~$f$ to $\Sigma f$. Then there is a commutative diagram 
\begin{equation} 
  \label{JEstab} 
  \diagram 
     \pi_{k-1}(SO(n-1))\rto^-{J}\dto^{i_{\ast}} & \pi_{n+k-2}(S^{n-1})\dto^{E} \\ 
     \pi_{k-1}(SO(n))\rto^-{J} & \pi_{n+k-1} (S^{n}). 
  \enddiagram 
\end{equation} 
If $k+2\leq n$ then Bott periodicity and the Freudenthal suspension theorem imply that we are in the stable range so $i_\ast$ and $E$ are isomorphisms. In this case, if 
\(\tau\colon\namedright{S^{k-1}}{}{SO(n)}\) 
represents a class in $\pi_{k-1}(SO(n))$, then $\tau\simeq i\circ\mu$ for some map 
\(\mu\colon\namedright{S^{k-1}}{}{SO(n-1)}\), 
so the commutativity of~(\ref{JEstab}) implies that $J(\tau)\simeq\Sigma J(\mu)$.

\begin{lemma}\label{taulemma2}
If $4\leq k+2\leq n$ then the composite 
\[
S^{n+k-2}\stackrel{j_2}{\larrow} S^{n-1}\vee S^{n+k-2}\stackrel{e}{\larrow} S^{n-1}\rtimes S^{k-1} \stackrel{\overline{\tau}}{\larrow}  S^{n-1}\rtimes S^{k-1}\stackrel{e^{-1}}{\larrow} S^{n-1}\vee S^{n+k-2}
\]
is homotopic to the composite
\[
S^{n+k-2} \stackrel{\sigma}{\larrow}S^{n+k-2}\vee S^{n+k-2} \stackrel{J(\mu) \vee 1}{\llarrow} S^{n-1}\vee S^{n+k-2}
\]
where $\sigma$ is the standard comultiplication and $\tau\simeq i\circ\mu$.  
\end{lemma}
\begin{proof} 
By the Hilton-Milnor Theorem, 
\[\pi_{n+k-2}(S^{n-1}\vee S^{n+k-2})\cong \pi_{n+k-2}(S^{n-1})\oplus \pi_{n+k-2}(S^{n+k-2}).\] 
Hence, to prove the lemma, it is equivalent to show that 
$p_{1}\circ e^{-1}\circ\overline{\tau}\circ e\circ j_{2}\simeq J(\mu)$ and 
$p_{2}\circ e^{-1}\circ\overline{\tau}\circ e\circ j_{2}$ is homotopic to the identity map. 

Starting with the $p_{2}$ case, by Lemmas \ref{taulemma} and \ref{gyrationpolemma}, the map $\widetilde{\tau}\colon\namedright{S^{n-1}\times S^{k-1}}{}{S^{n-1}\times S^{k-1}}$ restricts to the identity map on both sphere factors up to homotopy. In particular, it induces the identity homomorphism on $H^{t}(S^{n-1}\times S^{k-1})$ for $t<n+k-2$. The cup product structure then implies that $\widetilde{\tau}$ induces the identity homomorphism on $H^{n+k-2}(S^{n-1}\times S^{k-1})$ as well. Consequently, the reduction map 
\(\namedright{S^{n-1}\rtimes S^{k-1}}{\overline{\tau}}{S^{n-1}\rtimes S^{k-1}}\) 
also induces the identity map on cohomology. Thus $p_{2}\circ e^{-1}\circ\overline{\tau}\circ e\circ j_{2}$ 
induces the identity map in cohomology, and hence in homology, so the Hurewicz isomorphism implies that this composite is homotopic to the identity map on $S^{n+k-2}$. 

Now turn to the $p_{1}$ case. Consider the diagram
\[
\diagram 
  S^{n+k-1} \rto^<<<{\mathfrak{t}} \ddouble &    \Sigma( S^{n-1}\times S^{k-1} )\rto^{\Sigma \widetilde{\tau}} \dto^{ \Sigma q}  &  \Sigma (S^{n-1}\times S^{k-1}) \dto^{\Sigma q} \rto^<<<{\Sigma \pi_1} & S^{n} \ddouble \\
     S^{n+k-1} \rto^<<<{j} &   \Sigma (S^{n-1}\rtimes S^{k-1}) \rto^{\Sigma \overline{\tau}}               & S^{n-1}\rtimes S^{k-1} \rto^<<<<{\Sigma \overline{\pi}_1} &S^{n}.
  \enddiagram 
\]
The left square commutes by definition of $j$, the middle square homotopy commutes 
by Lemma~\ref{taureduce} and the right square commutes by definition of $\overline{\pi}_{1}$. 
Since $(\pi_1\circ \widetilde{\tau})(a, t)=\tau(t)(a)=(\vartheta \circ (1\times \tau))(a, t)$, we see that the composite along the top row satisfies $\Sigma \pi_1\circ  \Sigma \widetilde{\tau}\circ\mathfrak{t}= \Sigma\vartheta \circ \Sigma (1\times \tau)\circ\mathfrak{t}=J(\tau)$ by the definition of the $J$-homomorphism. The homotopy commutativity of the diagram therefore implies that 
$J(\tau)\simeq\Sigma\overline{\pi}_{1}\circ\Sigma\overline{\tau}\circ j$. 
Since $k+2\leq n$, by Lemma~\ref{jsusp}, $j\simeq\Sigma\overline{j}$. By Lemma~\ref{eprops}, 
$\overline{j}=e\circ j_{2}$. Therefore 
$J(\tau)\simeq\Sigma(\overline{\pi}_{1}\circ\overline{\tau}\circ e\circ j_{2})$. Since $n\geq k+2$, 
if tollows from stability in~(\ref{JEstab}) that  
$\overline{\pi}_1\circ \overline{\tau}\circ e\circ j_2\simeq J(\mu)$. 
By Lemma~\ref{eprops}, $\overline{\pi}_{1}\circ e\simeq p_{1}$, or 
equivalently, $\overline{\pi}_{1}\simeq p_{1}\circ e^{-1}$. Hence 
$J(\mu)\simeq \overline{\pi}_1\circ \overline{\tau}\circ e\circ j_2\simeq 
     p_{1}\circ e^{-1}\circ \overline{\tau}\circ e\circ j_2$, 
as required. 
\end{proof}

Define the map $\widehat{\tau}: S^{n-1}\vee  S^{n+k-2}\larrow S^{n-1}\vee  S^{n+k-2}$ by 
taking the wedge sum of the inclusion $j_1: S^{n-1}\larrow  S^{n-1}\vee  S^{n+k-2}$ and the composite $
S^{n+k-2} \stackrel{\sigma}{\larrow}S^{n+k-2}\vee S^{n+k-2} \stackrel{J(\mu) \vee 1}{\llarrow} S^{n-1}\vee S^{n+k-2}$ from Lemma~\ref{taulemma2}. It can be expressed in matrix form as:
\begin{equation}\label{taumatrix}
\widehat{\tau}=\begin{footnotesize}\begin{pmatrix} 
1 & J(\mu)\\
0 & 1 
\end{pmatrix}\end{footnotesize}:
S^{n-1}\vee  S^{n+k-2}\larrow S^{n-1}\vee  S^{n+k-2}.
\end{equation}

\begin{proposition}\label{gyrationprop}
Let $\tau: S^{k-1}\larrow SO(n)$ be a based map with $4\leq k+2\leq n$. Let $M$ be a path-connected $n$-dimensional Poincar\'{e} duality complex with a single top cell. Then there is a homotopy pushout
\begin{equation} 
  \label{gyrationpo3} 
  \diagram 
      S^{n-1}\vee  S^{n+k-2}  \dto^-{p_1} \rto^<<<{\widehat{\tau}}&  S^{n-1}\vee  S^{n+k-2}\rto^-{e} 
           & S^{n-1}\rtimes S^{k-1}\rto^<<<{f\rtimes 1} &\overline{M}\rtimes S^{k-1}  \dto \\ 
      S^{n-1}\xto[rrr] & & & \Gtau(M),
  \enddiagram 
\end{equation} 
where $p_1$ is the pinch map to the first wedge summand, $f$ is the attaching map for the top cell of~$M$, and the map $\widehat{\tau}$ is defined by \eqref{taumatrix}. 
In particular, if $J(\tau)$ is null homotopic, there is a homotopy pushout
\[ 
  \diagram 
      S^{n-1}\rtimes S^{k-1}\dto^-{\overline{\pi}_1}\rto^{f\rtimes 1} &\overline{M}\rtimes S^{k-1}  \dto \\ 
      S^{n-1}\rto & \Gtau(M),
  \enddiagram 
\]
and hence $\Gtau(M)\simeq \Gzero(M)$.
\end{proposition}
\begin{proof}
Comparing~(\ref{gyrationpo2}) and~(\ref{gyrationpo3}), to show that~(\ref{gyrationpo3}) 
is a homotopy pushout it suffices to show that 
$e\circ\widehat{\tau}\circ e^{-1}\simeq\overline{\tau}$ and $p_{1}\circ e^{-1}\simeq\overline{\pi}_{1}$. 
The latter holds by Lemma~\ref{eprops}. For the former,  
it is equivalent to show that $e\circ\overline{\tau}\circ e^{-1}\simeq\widehat{\tau}$. As both of 
these maps have domain $S^{n-1}\vee S^{n+k-2}$ it is equivalent to show that their restrictions 
to each wedge summand are homotopic. By Lemma~\ref{eprops}, $e\circ j_{1}\simeq i$, 
implying also that $e^{-1}\circ i\simeq j_{1}$. By definition of $\overline{\tau}$ we have 
$\overline{\tau}\circ i\simeq i$ and by definition of $\widehat{\tau}$ we have $\widehat{\tau}\circ j_{1}=j_{1}$. 
Therefore 
\[e^{-1}\circ\overline{\tau}\circ e\circ j_{1}\simeq e^{-1}\circ\overline{\tau}\circ i\simeq 
    e^{-1}\circ i\simeq j_{1}\simeq\widehat{\tau}\circ j_{1}.\] 
On the other hand, by Lemma~\ref{taulemma2} and the definition of $\widehat{\tau}$ we have 
\[e^{-1}\circ\overline{\tau}\circ e\circ j_{2}\simeq(J(\mu)\vee 1)\circ\sigma=\widehat{\tau}\circ j_{2}.\] 
Thus $e^{-1}\circ\overline{\tau}\circ e\simeq\widehat{\tau}$, as required. 

Next, if $J(\tau)$ is null homotopic then as $k+2\leq n$ we are in the stable range so  
$J(\tau)\simeq\Sigma J(\mu)$ and it follows that $J(\mu)$ is null homotopic. Thus, by its matrix definition, 
$\widehat{\tau}$ is homotopic to the identity map. Therefore precomposing~(\ref{gyrationpo3}) 
with $e^{-1}$ gives $(f\rtimes 1)$ along the top row and $p_{1}\circ e^{-1}\simeq\overline{\pi}_{1}$ by 
Lemma~\ref{eprops}, showing that $\Gtau(M)$ is the homotopy pushout of $f\rtimes 1$ 
and $\overline{\pi}_{1}$. But the homotopy pushout of these two maps is $\Gzero(M)$ 
by definition, so $\Gtau(M)\simeq\Gzero(M)$. 
\end{proof}

Proposition \ref{gyrationprop} implies that, for a given Poincar\'{e} duality complex $M$, the homotopy type of its homotopy gyration $\Gtau(M)$ is determined by the homotopy class of $J(\mu)$, the desuspended $J$-image of the framing $\tau$. Since the image of the stable $J$-homomorphism was computed by Adams \cite{A} and Quillen \cite{Q}, Proposition~\ref{gyrationprop} can be used to study the classification of the homotopy type of $\Gtau(M)$. Two interesting cases are as follows and will be used in Section~\ref{sec:cpn}. 

Let $k=2$ and $n\geq 4$. In this case, the $J$-homomorphism $J: \pi_1(SO(n))\larrow \pi_{n+1}(S^n)\cong \mathbb{Z}/2\mathbb{Z}\{\eta\}$ is an isomorphism and its image is generated by the complex Hopf element $\eta$. Therefore the map~$\widehat{\tau}$ in~\eqref{gyrationpo3} has matrix representation
\[
\widehat{\tau}(t)=
\begin{footnotesize}\begin{pmatrix} 
1 & t\cdot\eta\\
0 & 1 
\end{pmatrix}\end{footnotesize} 
: S^{n-1}\vee S^{n}\larrow S^{n-1}\vee S^{n}, \ \ \  t\in \mathbb{Z}/2\mathbb{Z}.
\] 
To remember the role of the complex Hopf element, let $\mathcal{G}^{t}_{\mathbb{C}}(M)$ be the homotopy gyration determined by $\widehat{\tau}(t)$. Proposition \ref{gyrationprop} immediately implies the following. 

\begin{corollary}\label{k=2gyrationcoro}
When $k=2$ and $n\geq 4$, the homotopy gyrations of $M$ have at most two homotopy types: 
$\mathcal{G}^{t}_{\mathbb{C}}(M)$ for $t=\mathbb{Z}/2\mathbb{Z}$, each satisfying a homotopy pushout
\[
  \diagram 
      S^{n-1}\vee  S^{n}  \dto^-{p_1} \rto^<<<{
    \big(  \begin{smallmatrix} 
1 & t\cdot\eta\\
0 & 1 
\end{smallmatrix}\big)
      }&  S^{n-1}\vee  S^{n}\rto^-{e} &  S^{n-1}\rtimes S^{1}\rto^<<<{f\rtimes 1}
       &\overline{M}\rtimes S^{1}  \dto \\ 
      S^{n-1}\xto[rrr] &&& \mathcal{G}^{t}_{\mathbb{C}}(M). 
  \enddiagram 
  \]
\end{corollary}
\vspace{-0.9cm}~$\qqed$\medskip 

Let $k=4$ and $n\geq 6$. In this case, the $J$-homomorphism $J: \pi_3(SO(n))\larrow \pi_{n+3}(S^n)\cong \mathbb{Z}/24\mathbb{Z}\{\nu\}$ is an epimorphism and its image is generated by the quaternionic Hopf element $\nu$. Hence the map $\widehat{\tau}$ in~\eqref{gyrationpo3} has matrix representation
\[
\widehat{\tau}(t)=
\begin{footnotesize}\begin{pmatrix} 
1 & t\cdot\nu\\
0 & 1 
\end{pmatrix}\end{footnotesize}
: S^{n-1}\vee S^{n+2}\larrow S^{n-1}\vee S^{n+2}, \ \ \  t\in \mathbb{Z}/24\mathbb{Z}.
\] 
To remember the role of the quaternionic Hopf element, let $\mathcal{G}^{t}_{\mathbb{H}}(M)$ be the homotopy gyration determined by $\widehat{\tau}(t)$. Proposition \ref{gyrationprop} 
immediately implies the following. 

\begin{corollary}\label{k=4gyrationcoro}
When $k=4$ and $n\geq 6$, the homotopy gyrations of $M$ have at most $24$ homotopy types: $\mathcal{G}^{t}_{\mathbb{H}}(M)$ for $t\in\mathbb{Z}/24\mathbb{Z}$, each satisfying a homotopy pushout
\[
  \diagram 
      S^{n-1}\vee  S^{n+2}  \dto^-{p_1} \rto^<<<{
    \big(  \begin{smallmatrix} 
1 & t\cdot\nu\\
0 & 1 
\end{smallmatrix}\big)
      }&  S^{n-1}\vee  S^{n+2}\rto^-{e} & S^{n-1}\rtimes S^{3}\rto^<<<{f\rtimes 1}
       &\overline{M}\rtimes S^{3}  \dto \\ 
      S^{n-1}\xto[rrr] &&& \mathcal{G}^{t}_{\mathbb{H}}(M). 
  \enddiagram 
  \]
\end{corollary}
\vspace{-0.9cm}~$\qqed$\smallskip

\section{A loop space decomposition of homotopy gyrations}
\label{sec:gyrationhgp}
In this section we prove a loop space decomposition for the trivial homotopy gyration and then use this to prove Theorem \ref{gyrationtypeintro}. This generalizes the corresponding result in \cite[Section 2]{HT1} and employs a similar argument.

\subsection{Homotopy gyrations with trivial framings} 
Let $M$ be a path-connected $n$-dimensional Poincar\'{e} Duality complex with a single top cell. 
By~\eqref{gyrationpo}, a homotopy gyration $\Gzero(M)$ of $M$ with trivial framing satisfies a 
homotopy pushout
\begin{equation} 
  \label{gyrationpocopy} 
  \diagram 
      S^{n-1}\times S^{k-1}\dto^-{\pi_1}\rto^{f\times 1} &\overline{M}\times S^{k-1}  \dto \\ 
      S^{n-1}\rto & \Gzero(M), 
  \enddiagram 
\end{equation} 
where $k\geq 2$ and $f$ is the atttaching map for the top cell of $M$.

\begin{lemma}\label{tlemma}
There is a homotopy commutative diagram
\begin{equation}\label{teq}
\begin{aligned}
  \xymatrix{ 
   S^{n-1}\times S^{k-1}\ar[d]^{\pi_1}\ar[r]^{f\times 1} & \overline{M}\times S^{k-1}\ar[d]\ar@/^/[ddr]^{\pi_1} &  \\ 
   S^{n-1}\ar[r]\ar@/_/[drr]_{f} & \Gzero(M) \ar[dr]^(0.4){t} & \\ 
   & & \overline{M}
   }  
   \end{aligned}
\end{equation} 
for some map $t$ that has a right homotopy inverse. 
\end{lemma}
\begin{proof}
Observe that the outer diagram commutes by the naturality of the projection $\pi_1$. Therefore, as 
the inner square is a homotopy pushout by \eqref{gyrationpocopy},  there is a map $t$ such that the two trianglular regions homotopy commute. In particular, since the projection $\pi_1: \overline{M}\times S^{k-1}\larrow \overline{M}$ has a right homotopy inverse, so does the map $t$.
\end{proof}

\begin{theorem} 
   \label{gyration0type} 
   Let $M$ be a path-connected $n$-dimensional Poincar\'{e} Duality complex with a single top cell and suppose that $k\geq 2$. Then there is a homotopy fibration 
   \[\nameddright{\Sigma^k H}{}{\Gzero(M)}{t}{\overline{M}}\] 
   where $H$ is the homotopy fibre of the attaching map 
   \(f: \namedright{S^{n-1}}{}{\overline{M}}\) for the top cell of $M$ and $t$ has a right homotopy inverse.
   In particular, this homotopy fibration splits after looping to give a homotopy equivalence
   \[\Omega\Gzero(M)\simeq\Omega \overline{M}\times\Omega\Sigma^k H.\] 
\end{theorem} 

\begin{proof} 
Start with the homotopy pushout~\eqref{gyrationpocopy}, compose maps with 
\(\namedright{\Gzero(M)}{t}{\overline{M}}\) 
and take homotopy fibres. Define the space $J$ by the homotopy fibration
\[
J\stackrel{}{\larrow} \Gzero(M) \stackrel{t}{\larrow} \overline{M}.
\] 
By definition of $H$ there is a homotopy fibration
\[
H\stackrel{j_H}{\larrow} S^{m-1}\stackrel{f}{\larrow} \overline{M} 
\] 
that defines the map $j_{H}$. Projections then lead to trivial homotopy fibrations 
\[\begin{split} 
    H\times S^{k-1}\stackrel{j_H\times 1}{\larrow}S^{m-1}\times S^{k-1} 
       \stackrel{f\circ\pi_1}{\larrow}\overline{M} & \\ 
    S^{k-1}\stackrel{i_2}{\larrow} \overline{M}\times S^{k-1}\stackrel{\pi_1}{\larrow} \overline{M}. & 
\end{split}\] 
As the four corners of the homotopy pushout~(\ref{gyrationpocopy}) have all been fibred over 
the common base space~$\overline{M}$, we obtain a homotopy commutative cube 
\[
  \spreaddiagramcolumns{-1pc}\spreaddiagramrows{-1pc} 
      \diagram
      H\times S^{k-1}\rrto^-{a}\drto^<<<<<<{b}\ddto^{j_H\times 1} & &S^{k-1}\dline^<<<{i_2}\drto & \\
      & H\rrto\ddto^>>>{j_H} & \dto & J\ddto \\
     S^{n-1}\times S^{k-1} \rline^-{f\times 1}\drto^{\pi_1} & \rto & \overline{M}\times S^{k-1}\drto & \\
      & S^{n-1}\rrto & & \Gzero(M)
   \enddiagram 
\] 
where the four sides are homotopy pullbacks, $a$ and $b$ are induced maps of fibres, and the 
bottom face is a homotopy pushout. By Lemma~\ref{cube}, the top face is also a homotopy pushout.

We want to identify the maps $a$ and $b$. The left face of the cube is obtained by the homotopy pullback of $j_H$ along the projection $\pi_1$, and therefore the map $b$ on the fibres is also the projection~$\pi_1$. The rear face of the cube is obtained by the product of the two homotopy pullbacks
\[
\diagram
H \rto^{} \dto^{j_H}  & \ast     \dto &&
S^{k-1} \rdouble \ddouble  & S^{k-1}  \ddouble \\
S^{n-1} \rto^{f}      & \overline{M},  &&
S^{k-1} \rdouble  & S^{k-1}, 
\enddiagram
\]
and therefore $a$ is the projection $\pi_2$.

Thus the homotopy pushout for $J$ in the top face of the cube 
is equivalent, up to homotopy, to that given by the projections 
\(\namedright{H\times S^{k-1}}{\pi_{1}}{H}\) 
and 
\(\namedright{H\times S^{k-1}}{\pi_{2}}{S^{k-1}}\). 
This homotopy pushout is in turn equivalent, up to homotopy, to the pushout given by the inclusions 
\(\namedright{H\times S^{k-1}}{1\times i}{H\times CS^{k-1}}\) 
and 
\(\namedright{H\times S^{k-1}}{i\times 1}{CH\times S^{k-1}}\), 
where $CS^{k-1}$ and $CH$ are the reduced cones on $S^{k-1}$ and $H$ respectively and the map $i$ 
in both cases is the inclusion into the base of the cone. The latter pushout is, by definition, 
the join $H\ast S^{k-1}$. There is a canonical homotopy equivalence $H\ast S^{k-1}\simeq\Sigma H\wedge S^{k-1}$. Hence $J\simeq\Sigma^{k} H$. 

To conclude, we have shown that there is a homotopy fibration $\Sigma^{k} H\larrow \Gzero (M) \stackrel{t}{\larrow} \overline{M}$. As the map $t$ has a right homotopy inverse by Lemma \ref{tlemma}, this homotopy fibration splits after looping. This proves the theorem.
\end{proof}

\subsection{Homotopy gyrations with general framings}
To study homotopy gyrations with general framings, let us recall the classical work of Adams \cite{A} and Quillen \cite{Q} on the image of the stable $J$-homomorphism. When $n\geq k+1$, they showed that the image of $J$ is
\begin{equation}\label{imJeq}
{\rm Im}\, J\cong \left\{\begin{array}{cc}
0 & k\equiv 3, 5, 6, 7 \ {\rm mod} \ 8, \\
\mathbb{Z}/2 & k\equiv 1, 2\ {\rm mod} \ 8, k\neq 1, \\
\mathbb{Z}/d_s& k=4s,
\end{array}\right.
\end{equation}
where $d_s$ is the denominator of $B_s/4s$ and $B_s$ is the $s$-th Bernoulli number defined by 
\[\frac{z}{e^z-1}=1-\frac{1}{2}z-\sum\limits_{s\geq 1} B_s \frac{z^{2s}}{(2s)!}.\] 
For each $k\geq 2$, let $\mathcal{P}_{k}$ be the set of prime numbers such that
\begin{equation}\label{pkdefeq}
\mathcal{P}_{k}= \left\{\begin{array}{cc}
\emptyset & k\equiv 3, 5, 6, 7 \ {\rm mod} \ 8, \\
\{2\} & k\equiv 1, 2\ {\rm mod} \ 8, k\neq 1, \\
\{p~|~p~{\rm divides}~d_s\}& k=4s,
\end{array}\right.
\end{equation}
In the $k=4s$ case, the set of primes $\mathcal{P}_k$ can be described explicitly. 
\begin{lemma}\label{Bslemma}
   $\mathcal{P}_{4s}=\{ {\rm prime}~p~|~(p-1)~{\rm divides}~2s\}$.
\end{lemma}
\begin{proof} 
By \cite[Theorem B.4]{MS}, a prime $p$ divides the denominator of $B_s/s$ 
if and only if it divides the denominator of $B_s$, while by \cite[Theorem B.3]{MS} the latter holds 
if and only if $p-1$ divides $2s$. 
In particular, $p=2$ always divides the denominator of $B_s/s$ and an odd prime divides $d_s$ if and only if $p-1$ divides $2s$. The lemma follows.
\end{proof} 

Localizing away from primes appearing in the image of the $J$-homomorphism then 
gives information about homotopy gyrations with nontrivial framings.

\begin{proposition}\label{local1prop} 
Let $M$ be a path-connected $n$-dimensional Poincar\'{e} Duality complex with a single top cell and 
consider the homotopy gyration $\Gtau(M)$ of $M$ with framing $\tau: S^{k-1}\larrow SO(n)$ 
defined by \eqref{gyrationpo}. If $4\leq k+2\leq n$ then localized away from $\mathcal{P}_k$ 
there is a homotopy equivalence
\[
\Gtau(M)\simeq \Gzero(M).
\]
\end{proposition}
\begin{proof} 
By Proposition \ref{gyrationprop}, the homotopy pushout defining $\Gtau(M)$ can be reformulated as 
the homotopy pushout \eqref{gyrationpo3}. After localizing away from $\mathcal{P}_{k}$ the image 
of the $J$-homomorphism is zero. Thus Proposition \ref{gyrationprop} implies that, as $J(\tau)$ is 
null homotopic, there is a homotopy equivalence $\Gtau(M)\simeq \Gzero(M)$. 
\end{proof}

We can now prove Theorem \ref{gyrationtypeintro}.
\begin{proof}[Proof of Theorem \ref{gyrationtypeintro}] 
The homotopy equivalence for $\Omega\Gzero(M)$ is given by Theorem \ref{gyration0type}. 
That the homotopy type of $\Gtau(M)$ for $n\geq k+2$ is determined by $J([\tau])$ is given 
by Propositions \ref{gyrationprop}. The homotopy equivalences in (1)-(3) follow from 
Proposition~\ref{local1prop} and Lemma \ref{Bslemma}.
\end{proof}

\section{Stabilizing by $\cp^{n}$ or $\hp^{n}$} 
\label{sec:cpn} 
Let $\mathbb{F}=\mathbb{C}$ or $\mathbb{H}$, and take $k=2$ or $4$ correspondingly. For $n\geq2$, in the complex or quaternionic case, let $\fp^{n}$ be the projective space of lines through the origin in $\mathbb{F}^n$. There is a homotopy cofibration 
\[\nameddright{S^{kn-1}}{g}{\fp^{n-1}}{j}{\fp^{n}}\] 
where $g$ attaches the $kn$-cell to $\fp^{n}$ and $j$ is the inclusion. 
Let $M$ be a path-connected $kn$-dimensional Poincar\'{e} Duality complex with a single top cell. There is a homotopy cofibration  
\[\nameddright{S^{kn-1}}{f}{\overline{M}}{i}{M}\] 
where $f$ attaches the $kn$-cell to $M$ and $i$ is the inclusion. By definition of the connected sum 
there is a homotopy pushout 
\begin{equation} 
   \label{MCPpo} 
   \diagram 
       S^{kn-1}\rto^-{f}\dto^{g} &  \overline{M}\dto \\ 
       \fp^{n-1}\rto & M\conn\fp^{n}. 
    \enddiagram 
\end{equation} 

Our goal in this section is to prove Theorem \ref{fibfpintro}. This requires several steps along the way. 
Let 
\[\overline{h}\colon\namedright{\fp^{n-1}}{}{\fp^{\infty}}\] 
be the standard inclusion. Observe that $\overline{h}$ factors through the inclusion 
\(\namedright{\fp^{n-1}}{j}{\fp^{n}}\), 
implying that $\overline{h}\circ g$ is null homotopic since it factors through the composite $j\circ g$ of two consecutive maps in a homotopy cofibration. 
Combining this with the constant map to the basepoint, 
\(\namedright{\overline{M}}{\ast}{\fp^{\infty}}\), 
from the homotopy pushout~(\ref{MCPpo}) we obtain a pushout map $h'$ that makes the following 
diagram homotopy commute 
\begin{equation} 
\begin{aligned}
  \label{CPconnpo} 
  \xymatrix{ 
   S^{kn-1}\ar[r]^{f}\ar[d]^{g} & \overline{M}\ar[d]\ar@/^/[ddr]^{\ast} &  \\ 
   \fp^{n-1}\ar[r]\ar@/_/[drr]_{\overline{h}} & M\conn\fp^{n}\ar@{.>}[dr]^{h'} & \\ 
   & & \fp^{\infty}. }  
   \end{aligned}
\end{equation} 
Observe that the inclusion of the bottom cell $S^{k-1}\hookrightarrow \Omega \fp^{n-1}$ 
composes with $\Omega\overline{h}$ to give the inclusion of the bottom cell 
$S^{k-1}\hookrightarrow\Omega\fp^{\infty}$. As $\Omega\fp^{\infty}\simeq S^{k-1}$, the latter inclusion 
is a homotopy equivalence, implying that $\Omega\overline{h}$ has a right homotopy inverse. 
Now as $\Omega\overline{h}$ has a right homotopy inverse, the homotopy 
commutativity of the lower triangle in~(\ref{CPconnpo}) immediately implies the following. 

\begin{lemma}\label{h'lemma}
The loop map $\Omega h': \Omega (M\conn\fp^{n})\larrow \Omega \fp^{\infty}$ has a right homotopy inverse. ~$\qqed$
\end{lemma}
To prove Theorem \ref{fibfpintro} we need to study the homotopy fibre of $h'$.

\subsection{Analyzing the homotopy fibre of $h'$}
Recall the Hopf fibration $S^{k-1}\larrow S^{kn-1}\stackrel{g}{\larrow} \fp^{n-1}$ is a principal $S^{k-1}$-bundle with an associated action of $S^{k-1}$ on $S^{kn-1}$ through a map
\[
\vartheta: S^{kn-1}\times S^{k-1}\larrow S^{n-1}.
\]
\begin{lemma}\label{CPconncubelemma}
There is a homotopy commutative cube
\begin{equation} 
  \label{CPconncube0} 
  \spreaddiagramcolumns{-1pc}\spreaddiagramrows{-1pc} 
      \diagram
      S^{kn-1}\times S^{k-1}\rrto^-{f\times 1}\drto^<<<<<<{\vartheta}\ddto^{\pi_1} & & \overline{M}\times S^{k-1}\dline^<<<{\pi_1}\drto & \\
      & S^{kn-1}\rrto\ddto^>>>{g} & \dto & F\ddto \\
      S^{kn-1}\rline^-(0.6){f}\drto^{g} & \rto & \overline{M}\drto & \\
      & \fp^{n-1}\rrto & & M\conn\fp^{n},
   \enddiagram 
\end{equation} 
where $F$ is the homotopy fibre of $h': M\conn\fp^{n}\larrow \fp^{\infty}$, the four sides are homotopy pullbacks, and the bottom and top faces are homotopy pushouts. 
\end{lemma}
\begin{proof}
The lemma will be proved in two steps.

\smallskip 
\noindent 
\textit{Step 1: The cube}. 
Start with the homotopy pushout~(\ref{MCPpo}), compose maps with 
\(\namedright{M\conn\fp^{n}}{h'}{\fp^{\infty}}\), 
and take homotopy fibres. By definition of $F$ there is a homotopy fibration  
\[\nameddright{F}{}{M\conn\fp^{n}}{h'}{\fp^{\infty}}.\]
The Hopf fibration implies there is the standard homotopy fibration 
\[\nameddright{S^{kn-1}}{g}{\fp^{n-1}}{\overline{h}}{\fp^{\infty}.}\] 
Since $\Omega\fp^{\infty}\simeq S^{k-1}$, there are trivial homotopy fibrations 
\[\begin{split} 
     \nameddright{\overline{M}\times S^{k-1}}{\pi_1}{\overline{M}}{\ast}{\fp^{\infty},} \\ 
     \nameddright{S^{kn-1}\times S^{k-1}}{\pi_1}{S^{kn-1}}{\ast}{\fp^{\infty}}. & 
  \end{split}\] 
As the four corners of the homotopy pushout~(\ref{CPconnpo}) have all been 
fibred over the common base space $\fp^{\infty}$,  we obtain a homotopy commutative cube 
\[
  \label{CPconncube} 
  \spreaddiagramcolumns{-1pc}\spreaddiagramrows{-1pc} 
      \diagram
      S^{kn-1}\times S^{k-1}\rrto^-{a}\drto^<<<<<<{b}\ddto^{\pi_1} & & \overline{M}\times S^{k-1}\dline^<<<{\pi_1}\drto & \\
      & S^{kn-1}\rrto\ddto^>>>{g} & \dto & F\ddto \\
      S^{kn-1}\rline^-{f}\drto^{g} & \rto & \overline{M}\drto & \\
      & \fp^{n-1}\rrto & & M\conn\fp^{n}
   \enddiagram 
\] 
where the four sides are homotopy pullbacks, $a$ and $b$ are induced maps of fibres, and the 
bottom face is a homotopy pushout. By Lemma~\ref{cube}, the top face is also a homotopy 
pushout. To complete the proof of the lemma, it remains to identify the maps $a$ and $b$.
\smallskip 

\noindent 
\textit{Step 2: Identifying $a$ and $b$}. The rear face of the cube is induced by 
mapping $\overline{M}$ trivially to $\fp^{\infty}$, so $a$ is the product map $f\times 1$, where $1$ is 
the identity map on $S^{k-1}$. For the left face of the cube, the principal action $\vartheta: S^{kn-1}\times S^{k-1}\larrow S^{kn-1}$ of the Hopf bundle $S^{k-1}\larrow S^{kn-1}\stackrel{g}{\larrow} \fp^{n-1}$ satisfies the canonical commutative diagram
\[\diagram 
       S^{kn-1}\times S^{k-1}\rto^<<<{\vartheta}\dto^{\pi_1} & S^{kn-1}\dto^{g} \\ 
       S^{kn-1}\rto^-{g} & \fp^{n-1},
  \enddiagram\] 
  which means the bundle projection $g$ projects an orbit to a point. 
It is clearly a pullback of principal $S^{k-1}$-bundles, and thus extends further to a homotopy fibration diagram involving the classifying space $BS^{k-1}\simeq \fp^{\infty}$:
\[
\diagram 
     S^{kn-1}\times S^{k-1}\rto^<<<{\pi_1}\dto^{\vartheta} & S^{kn-1}\rto^-{\ast}\dto^{g} 
           & \fp^{\infty}\ddouble \\ 
      S^{kn-1}\rto^{g} & \fp^{n-1}\rto^-{\overline{h}} & \fp^{\infty}.
  \enddiagram\] 
In particular, as taking homotopy fibres with respect to the right square is what induces 
the left face of the cube, we obtain $b=\vartheta$.
\end{proof}

The top face of~\eqref{CPconncube0} can be refined. Note that in the Hopf fibration \(\nameddright{S^{k-1}}{\partial}{S^{kn-1}}{g}{\fp^{n-1}}\) with $n\geq 2$, the connecting map $\partial$ is null homotopic. As the restriction of the principal action $\vartheta$ to $S^{k-1}$ is $\partial$, the null homotopy for $\partial$ implies that $\vartheta$ factors as a composite 
\begin{equation}\label{thetaeq}
\vartheta: \ \ \nameddright{S^{kn-1}\times S^{k-1}}{q}{S^{kn-1}\rtimes S^{k-1}}{\theta}{S^{kn-1}},
\end{equation}
where $q$ collapses $S^{k-1}$ to a point and $\theta$ is an induced quotient map. There may be a choice in the quotient map $\theta$ that factors $\vartheta$; selecting a choice in each of the complex 
and quaternionic cases will be discussed momentarily, but for now the following lemma holds for any choice. 

\begin{lemma} 
\label{CPthetapo} 
 There is a homotopy pushout
 \[\diagram 
       S^{kn-1}\rtimes S^{k-1}\rto^-{f\rtimes 1}\dto^{\theta} 
            & \overline{M}\rtimes S^{k-1}\dto \\ 
       S^{kn-1}\rto & F. 
  \enddiagram\] 
\end{lemma}
\begin{proof}
Consider the iterated homotopy pushout diagram 
\[
\diagram
    S^{kn-1}\times S^{k-1} \dto^{f\times 1} \rto^{q}  &  S^{kn-1}\rtimes S^{k-1} \dto^{f\rtimes 1} \rto^<<<{\theta}  & S^{kn-1} \dto \\
   \overline{M}\times S^{k-1} \rto^{q}  & \overline{M}\rtimes S^{k-1} \rto^{}  & F', 
\enddiagram
\]
where the left square commutes by the naturality of the half-smash and is a homotopy 
pushout because a common copy of $S^{k-1}$ has been collapsed out from the left side, 
and the space $F'$ is defined as the homotopy pushout of $\theta$ and $f\rtimes 1$. Since the composite $\theta\circ q$ along the top row is homotopic to $\vartheta$, the outer diagram is the homotopy pushout in the top face of~(\ref{CPconncube0}). Therefore, $F'\simeq F$ and the lemma follows.
\end{proof}
Lemma \eqref{CPthetapo} is crucial for proving Theorem \ref{fibfpintro}, with a good choice of the reduced action $\theta$ playing an important role. We discuss the complex and quaternionic cases separately.

\subsection{The complex case} In this case, $k=2$, $\mathbb{F}=\mathbb{C}$ and $M$ is $2n$-dimensional with $n\geq 2$. 

\begin{lemma}{\cite[Lemma 10.1]{HT2}}
   \label{thetabarfactor} 
   In the complex case, the reduced action $\theta$ in \eqref{thetaeq} may be chosen so that there is a homotopy commutative diagram 
   \[
   \label{HTdgrm}\diagram 
         S^{2n-1}\rtimes S^{1}\rto^-{\theta}\dto^{e^{-1}} & S^{2n-1}\ddouble \\ 
         S^{2n-1}\vee S^{2n}\rto^-{1+n\cdot\eta} & S^{2n-1}.
      \enddiagram
      \]
\end{lemma}  
\vspace{-0.9cm}~$\qqed$\bigskip 

\begin{remark} 
Because the choice of homotopy equivalence 
\(\namedright{S^{2n-1}\rtimes S^{1}}{}{S^{n-1}\vee S^{n}}\) 
in Proposition~\ref{gyrationprop} is specific, it is worth checking that it is the same choice made 
in~\cite[Lemma~10.1]{HT2}. In that paper, the homotopy equivalence is the inverse of a homotopy 
equivalence defined prior to the statement of Theorem 9.1. That was defined using the natural 
homotopy equivalence 
\(\namedright{\Sigma A\vee(\Sigma A\wedge B)}{}{(\Sigma A)\rtimes B}\) 
determined by the inclusion 
\(i\colon\namedright{\Sigma A}{}{(\Sigma A)\rtimes B}\) 
and the composite 
\(j\colon\nameddright{\Sigma A\wedge B}{}{\Sigma(A\times B)}{\Sigma q} 
    {\Sigma(A\rtimes B)\simeq(\Sigma A)\rtimes B}\) 
where the left map is from the Hopf construction and $q$ is the quotient map to the half-smash. 
These maps correspond exactly to the maps $i$ and $j$ used to define $e$ in Section~\ref{sec:gyration}. 
\end{remark}

Recall the homotopy gyrations $\mathcal{G}^t_{\mathbb{C}}(M)$ for $t\in\mathbb{Z}/2\mathbb{Z}$ in Corollary \ref{k=2gyrationcoro}. For convenience, if $n$ is a positive integer and $n\equiv t~{\rm mod}~2$ then write $\mathcal{G}^n_{\mathbb{C}}(M)$ for $\mathcal{G}^t_{\mathbb{C}}(M)$. 

\begin{theorem} 
   \label{fibcp} 
   Let $M$ be a path-connected $2n$-dimensional Poincar\'{e} Duality complex with a single top cell and $n\geq 2$. Then there is a homotopy fibration 
   \[\nameddright{\mathcal{G}^n_{\mathbb{C}}(M)}{}{M\conn\cp^{n}}{h'}{\cp^{\infty}}\] 
   where $\Omega h'$ has a right homotopy inverse. Consequently, the homotopy fibration splits after looping to give a homotopy equivalence
   \[
   \Omega (M\conn\cp^{n})\simeq S^1\times \Omega\mathcal{G}^n_{\mathbb{C}}(M).
   \]
\end{theorem} 
\begin{proof}
Consider the homotopy fibration 
\(\nameddright{F}{}{M\conn\cp^{n}}{h'}{\cp^{\infty}}\). By Lemma \ref{h'lemma}, $\Omega h'$ has a right homotopy inverse, implying that there is a homotopy equivalence 
$\Omega(M\conn\cp^{n})\simeq S^{1}\times\Omega F$. To complete the proof of the theorem, 
it remains to show that $F\simeq \mathcal{G}^n_\mathbb{C}(M)$.

Consider the diagram 
\[\diagram  
        S^{2n-1}\vee S^{2n} 
       \rto^-{
       \big(\begin{smallmatrix} 1 & n\cdot\eta \\
                                         0  &  1
              \end{smallmatrix}\big)
              }            \dto^{p_{1}} 
          &  S^{2n-1}\vee S^{2n}\rto^-{e}\dto^-{
          (\begin{smallmatrix}   1, &\! n\cdot\eta
           \end{smallmatrix})
          } & S^{2n-1}\rtimes S^{1}\rto^<<<{f\rtimes 1} & \overline{M}\rtimes S^{1}\dto \\ 
       S^{2n-1}\rdouble & S^{2n-1}\rrto && F, 
  \enddiagram\] 
  where $(1, n\cdot\eta)$ is the matrix expression of $1+n\cdot\eta$. Since $\eta$ is of order $2$, we have 
  \[
    \begin{pmatrix} 
     1 & n\cdot\eta
      \end{pmatrix}
   \cdot 
  \begin{pmatrix} 
  1 & n\cdot\eta \\
  0  &  1
  \end{pmatrix}
  =  \begin{pmatrix} 
  1 &0
    \end{pmatrix},
\]
which implies that the left square homotopy commutes. Further, 
as both horizontal maps in the left square are homotopy equivalences, it is a homotopy pushout. Since 
$\theta\simeq  (1+n\cdot\eta)\circ e^{-1}=(1, n\cdot\eta)\circ e^{-1}$ 
by Lemma~\ref{thetabarfactor}, precomposing the right rectangle with $e^{-1}$ gives 
the homotopy pushout in Lemma~\ref{CPthetapo}. As $e^{-1}$ is a homotopy equivalence, 
this implies that the right rectangle is itself a homotopy pushout. Thus the outer diagram is the 
juxtaposition of homotopy pushouts so it is a homotopy pushout. 
 By Corollary \ref{k=2gyrationcoro}, the homotopy pushout of 
$(f\rtimes 1)\circ e\circ \big(\begin{smallmatrix} 1 & n\cdot\eta \\ 0  &  1\end{smallmatrix}\big)$ and $p_{1}$ 
is $\mathcal{G}^{n}_{\mathbb{C}}(M)$. Thus $F\simeq\mathcal{G}^{n}_\mathbb{C}(M)$. 
\end{proof}

\subsection{The quaternionic case} In this case, $k=4$, $\mathbb{F}=\mathbb{H}$ and $M$ is $4n$-dimensional with $n\geq 2$. The following lemma is the quaternionic version of Lemma \ref{thetabarfactorH}.

\begin{lemma}
   \label{thetabarfactorH} 
   In the quaternionic case, the reduced action $\theta$ in \eqref{thetaeq} may be chosen so that there is a homotopy commutative diagram 
   \[
   \label{HTdgrmH}\diagram 
         S^{4n-1}\rtimes S^{3}\rto^-{\theta}\dto^{e^{-1}} & S^{4n-1}\ddouble \\ 
         S^{4n-1}\vee S^{4n+2}\rto^-{1\pm n\cdot\nu} & S^{4n-1}.
      \enddiagram
      \]
\end{lemma}  
\begin{proof}
The proof is similar to that of Lemma \ref{thetabarfactor} in \cite[Section 9-10]{HT2}. By \cite[Lemma~9.6 and Remark 9.7]{HT2}, for the principal fibration $S^3\larrow S^{4n-1}\stackrel{g}{\larrow}\hp^{n-1}$, the reduced action $\theta$ can be chosen so that there is a homotopy commutative diagram  
\[
   \diagram 
         S^{4n-1}\rtimes S^{3}\rto^-{\theta}\dto^{e^{-1}} & S^{4n-1}\dto^{g} \\ 
         S^{4n-1}\vee S^{4n+2}\rto^-{g+[i,g]} & \hp^{n-1},
      \enddiagram
      \]
      where $[i, g]: S^3\wedge S^{4n-1}\larrow \hp^{n-1}$ is the Whitehead product of $g$ and the inclusion $S^4\stackrel{i}{\hookrightarrow}\hp^{n-1}$ of the bottom cell. In~\cite{BJS} it was shown that $[i,g]$ is homotopic to the composite 
\(\lnameddright{S^{4n+2}}{(\pm n)\cdot\nu}{S^{4n-1}}{g}{\mathbb{H}P^{n-1}}\). Hence, the diagram implies that $g\circ \theta\simeq g+[i,g]\simeq g\circ (1\pm n\cdot\nu)$. However, in the homotopy fibration 
\(\nameddright{S^{3}}{\delta}{S^{4n-1}}{g}{\mathbb{H}P^{n-1}}\) 
the map $\delta$ is null homotopic, so $g$ induces an injection 
\(\namedright{[S^{m}, S^{4n-1}]}{g_{\ast}}{[S^{m},\mathbb{H}P^{n-1}]}\) 
for all $m\geq 1$. Therefore $g\circ\theta\simeq g\circ(1\pm n\cdot\nu)$ implies that 
$\theta\simeq 1\pm n\cdot\nu$. This proves the lemma.
\end{proof}

Recall the homotopy gyrations $\mathcal{G}^t_{\mathbb{H}}(M)$ for $t\in\mathbb{Z}/24\mathbb{Z}$ in Corollary \ref{k=4gyrationcoro}. The ambiguity of the sign in Lemma~\ref{thetabarfactorH} 
will unfortunately need to be carried along. Let $(-1)^{s(n)}$ be the actual (unknown) value of the 
sign, so that $\theta\simeq(1+(-1)^{s(n)}n\cdot\nu)\circ e^{-1}$. As the proof of the next lemma 
will introduce a multiplication by $-1$, let $\overline{n}=(-1)^{s(n)+1}n$. 
If $\overline{n}\equiv t~{\rm mod}~24$ then write 
$\mathcal{G}^{\overline{n}}_{\mathbb{H}}(M)$ for $\mathcal{G}^{t}_{\mathbb{H}}(M)$. 

\begin{theorem} 
   \label{fibhp} 
   Let $M$ be a path-connected $4n$-dimensional Poincar\'{e} Duality complex with a single top cell and $n\geq 2$. Then there is a homotopy fibration 
   \[\nameddright{\mathcal{G}^{\overline{n}}_{\mathbb{H}}(M)}{}{M\conn\hp^{n}}{h'}{\hp^{\infty}}\] 
   where $\Omega h'$ has a right homotopy inverse. Consequently, the homotopy fibration splits after looping to give a homotopy equivalence
   \[
   \Omega (M\conn\hp^{n})\simeq S^3\times \Omega\mathcal{G}^{\overline{n}}_{\mathbb{H}}(M).
   \]
\end{theorem} 
\begin{proof}
Consider the homotopy fibration 
\(\nameddright{F}{}{M\conn\hp^{n}}{h'}{\hp^{\infty}}\). By Lemma \ref{h'lemma}, $\Omega h'$ has a right homotopy inverse, implying that there is a homotopy equivalence 
$\Omega(M\conn\hp^{n})\simeq S^{1}\times\Omega F$. To complete the proof of the theorem, it remains to show that $F\simeq \mathcal{G}^{\overline{n}}_\mathbb{H}(M)$.

Consider the diagram 
\[\diagram  
        S^{4n-1}\vee S^{4n+2} 
       \rrto^-{
       \big(\begin{smallmatrix} 1 &  \mp n\cdot\nu \\
                                         0  &  1
              \end{smallmatrix}\big)
              }            \dto^{p_{1}} 
          & &  S^{4n-1}\vee S^{4n+2}\rto^-{e}\dto^-{
          (\begin{smallmatrix}   1, & \!\pm n\cdot\nu
           \end{smallmatrix})
          } & S^{4n-1}\rtimes S^{3}\rto^<<<{f\rtimes 1} & \overline{M}\rtimes S^{3}\dto \\ 
       S^{4n-1}\rrdouble  & & S^{4n-1}\rrto & & F, 
  \enddiagram\] 
  where $(1, \pm n\cdot\nu)$ is the matrix expression of $1\pm n\cdot\nu$. We have 
  \[
    \begin{pmatrix} 
     1 & \pm n\cdot\eta
      \end{pmatrix}
   \cdot 
  \begin{pmatrix} 
  1 & \mp n\cdot\eta \\
  0  &  1
  \end{pmatrix}
  =  \begin{pmatrix} 
  1 &0
    \end{pmatrix},
\]
which implies that the left square homotopy commutes. Further, as both horizontal maps in the left square are homotopy equivalences, it is a homotopy pushout. Since 
$\theta\simeq  (1\pm n\cdot\nu)\circ e^{-1}=(1, \pm n\cdot\nu)\circ e^{-1}$ 
by Lemma~\ref{thetabarfactorH}, precomposing the right rectangle with $e^{-1}$ gives the homotopy pushout in Lemma~\ref{CPthetapo}. As $e^{-1}$ is a homotopy equivalence, this 
implies that the right rectangle is itself a homotopy pushout. Thus the outer diagram is the juxtaposition 
of homotopy pushouts so it is a homotopy pushout. By Corollary \ref{k=4gyrationcoro}, the homotopy pushout of 
$(f\rtimes 1)\circ e\circ\big(\begin{smallmatrix} 1 & \mp n\cdot\nu \\ 0  &  1\end{smallmatrix}\big)$ and $p_{1}$ 
is $\mathcal{G}^{\mp n}_{\mathbb{H}}(M)$. Thus 
$F\simeq\mathcal{G}^{\mp n}_\mathbb{H}(M)$. Rewriting to address the ambiguity 
of sign gives $F\simeq\mathcal{G}^{\overline{n}}_\mathbb{H}(M)$. 
\end{proof}

\begin{proof}[Proof of Theorem~\ref{fibfpintro}] 
Theorems~\ref{fibcp} and~\ref{fibhp} proves parts (1) and (2) respectively. 
\end{proof}

\section{Stabilizing by other $T$, Part I} 
\label{sec:Tstab} 

Let $M$ be a path-connected $n$-dimensional Poincar\'{e} Duality complex with a single top cell. Let~$T$ 
be an $(m-1)$-connected $n$-dimensional Poincar\'{e} Duality complex with $2\leq m<n$. 
There are homotopy cofibrations 
\[\nameddright{S^{n-1}}{f}{\overline{M}}{i}{M}\] 
\[\nameddright{S^{n-1}}{g}{\overline{T}}{j}{T}\] 
where $f$ and $g$ attach the $n$-cell to $M$ and $T$ respectively and $i$ and $j$ are inclusions. 
By definition of the connected sum, there is a homotopy pushout 
\begin{equation} 
   \label{MTpo} 
   \diagram 
       S^{n-1}\rto^-{f}\dto^{g} &  \overline{M}\dto \\ 
       \overline{T}\rto & M\conn T. 
    \enddiagram 
\end{equation} 

A hypothesis is introduced on $T$ that matches the hypothesis introduced in~\cite{T3} 
when considering ``inert" attaching maps. Suppose that there is a map 
\(\namedright{S^{m}}{s}{T}\) 
with a left homotopy inverse 
\(\namedright{T}{h}{S^{m}}\). 
Since $m<n$ and~$\overline{T}$ is the $(n-1)$-skeleton of $T$, the map 
\(\namedright{S^{m}}{s}{T}\) 
factors as a composite 
\[\nameddright{S^{m}}{\overline{s}}{\overline{T}}{j}{T}\] 
for some map $\overline{s}$. Let $\overline{h}$ be the composite 
\[\overline{h}\colon\nameddright{\overline{T}}{j}{T}{h}{S^{m}}.\] 
Then $\overline{h}\circ\overline{s}=h\circ j\circ\overline{s}\simeq h\circ s$ is homotopic to the 
identity map. Thus $\overline{h}$ has a right homotopy inverse. Observe that $\overline{h}\circ g$ 
is null homotopic because it factors through the composite $j\circ g$ of two consecutive maps 
in a homotopy cofibration. Combining this with the constant map to the basepoint, 
\(\namedright{\overline{M}}{\ast}{S^{m}}\), 
from the homotopy pushout~(\ref{MTpo}) we obtain a pushout map $h'$ that makes the following 
diagram homotopy commute 
\begin{equation} 
\begin{aligned}
  \label{connpo} 
  \xymatrix{ 
   S^{n-1}\ar[r]^{f}\ar[d]^{g} & \overline{M}\ar[d]\ar@/^/[ddr]^{\ast} &  \\ 
   \overline{T}\ar[r]\ar@/_/[drr]_{\overline{h}} & M\conn T\ar@{.>}[dr]^{h'} & \\ 
   & & S^{m}. }  
   \end{aligned}
\end{equation} 
Note that as $\overline{h}$ has a right homotopy inverse, the homotopy commutativity 
of the lower triangle in~(\ref{connpo}) immediately implies the following.
\begin{lemma}\label{h'lemmaT}
The map $h': M\conn T\larrow S^m$ has a right homotopy inverse. ~$\qqed$
\end{lemma}

The following theorem restates Theorem \ref{fibidintro} with a consequent 
loop space decomposition stated explicitly. Two proofs are provided. The first proof, presented in this section, follows the strategy for proving Theorem \ref{fibfpintro} in Section \ref{sec:cpn} but with substantial changes in the details. The second proof will be given in Section \ref{sec:Tstab2}.
\begin{theorem} 
   \label{fibid} 
   Let $M$ be a path-connected $n$-dimensional Poincar\'{e} Duality complex with a single top cell. Let $T$ 
   be an $(m-1)$-connected $n$-dimensional Poincar\'{e} Duality complex with $2\leq m<n$. If there is a map 
   \(\namedright{S^{m}}{s}{T}\) 
   that has a left homotopy inverse 
   \(\namedright{T}{h}{S^{m}}\) 
   then there is a homotopy fibration 
   \[\nameddright{E\vee (\overline{M}\rtimes\Omega S^{m})}{}{M\conn T}{h'}{S^{m}},\] 
   where $E$ is the homotopy fibre of $h$, and $h'$ has a right homotopy inverse. Therefore, the homotopy fibration splits after looping to give a homotopy equivalence
\[
\Omega (M\conn T) \simeq \Omega S^m \vee \Omega (E\vee (\overline{M}\rtimes\Omega S^{m})).
\]
\end{theorem}

\subsection{Analyzing the homotopy fibre of $h'$} 
Define the space $G$ and the map $q_{G}$ by the homotopy fibration
\(\nameddright{G}{q_{G}}{\overline{T}}{\overline{h}}{S^{m}}\). 

\begin{lemma}\label{conncubelemma}
There is a lift 
\(\namedright{S^{n-1}}{\mathfrak{g}}{G}\) 
of 
\(\namedright{S^{n-1}}{g}{\overline{T}}\)  
that results in a homotopy commutative cube
\begin{equation} 
  \label{conncube0} 
  \spreaddiagramcolumns{-1pc}\spreaddiagramrows{-1pc} 
      \diagram
      S^{n-1}\times\Omega S^{m}\rrto^-{f\times 1}\drto^<<<<<<{\vartheta_G\circ (\mathfrak{g}\times 1)}\ddto^{\pi_1} & & \overline{M}\times\Omega S^{m}\dline^{\pi_1} \drto & \\
      & G\rrto\ddto & \dto & F\ddto \\
      S^{n-1}\rline^-{f}\drto^{g} & \rto & \overline{M}\drto & \\
      & \overline{T}\rrto & & M\conn T,
   \enddiagram 
\end{equation} 
where $F$ is the homotopy fibre of $h': M\conn T\larrow  S^m$, the four sides are homotopy pullbacks, and the bottom and top faces are homotopy pushouts. 
\end{lemma}
\begin{proof}
The lemma will be proved in two steps.

\smallskip 
\noindent 
\textit{Step 1: The cube}. 
Start with the homotopy pushout~(\ref{MTpo}), compose maps with 
\(\namedright{M\conn T}{h'}{S^{m}}\), 
and take homotopy fibres. By the definitions of $F$ and $G$ there are homotopy fibrations 
\[\begin{split} 
     \nameddright{F}{}{M\conn T}{h'}{S^{m}} &  \\ 
     \nameddright{G}{}{\overline{T}}{\overline{h}}{S^{m}}. & 
  \end{split}\] 
There are also trivial homotopy fibrations 
\[\begin{split} 
     \nameddright{\overline{M}\times\Omega S^{m}}{\pi_1}{\overline{M}}{\ast}{S^{m}} \\ 
     \nameddright{S^{n-1}\times\Omega S^{m}}{\pi_1}{S^{n-1}}{\ast}{S^{m}}. & 
  \end{split}\] 
As the four corners of the homotopy pushout~(\ref{MTpo}) have all been fibred over 
the common base space~$S^{m}$, we obtain a homotopy commutative cube 
\begin{equation} 
  \label{conncube} 
  \spreaddiagramcolumns{-1pc}\spreaddiagramrows{-1pc} 
      \diagram
      S^{n-1}\times\Omega S^{m}\rrto^-{a}\drto^{b}\ddto^{\pi_1} & & \overline{M}\times\Omega S^{m}\dline^{\pi_1} \drto & \\
      & G\rrto\ddto & \dto & F\ddto \\
      S^{n-1}\rline^-{f}\drto^{g} & \rto & \overline{M}\drto & \\
      & \overline{T}\rrto & & M\conn T
   \enddiagram 
\end{equation} 
where the four sides are homotopy pullbacks, $a$ and $b$ are induced maps of fibres, and the 
bottom face is a homotopy pushout. By Lemma~\ref{cube}, the top face is also a 
homotopy pushout. To complete the proof of the lemma, it remains to identify the maps $a$ and $b$.
\smallskip 

\noindent 
\textit{Step 2: Identifying $a$ and $b$}. 
In general, a homotopy fibration 
\(\nameddright{F}{q_{F}}{E}{}{B}\) 
has a connecting map 
\(\partial_{F}\colon\namedright{\Omega B}{}{F}\) 
and a homotopy action 
\[\vartheta_{F}\colon\namedright{F\times\Omega B}{}{F}\] 
whose restriction to $F$ is homotopic to the identity map and whose restriction to $\Omega B$ 
is homotopic to $\partial_{F}$. A standard property is that there is a homotopy commutative square 
\begin{equation} 
   \label{actionproj} 
   \diagram 
       F\times\Omega B\rto^-{\vartheta_{F}}\dto^{\pi_{1}} & F\dto^{q_{F}} \\ 
       F\rto^-{q_{F}} & E. 
   \enddiagram 
\end{equation} 
Moreover, these properties are natural for maps between homotopy fibrations. 

In our case, the rear face of the cube is induced by mapping $\overline{M}$ trivially to $S^{m}$, 
so~$a$ is the product map $f\times 1$, where $1$ is the identity map on $\Omega S^{m}$. Explicitly, 
if $H$, $K$ and $a'$ are defined by the homotopy fibration diagram 
\[\diagram 
      \Omega S^{m}\rto^-{\partial_{H}}\ddouble & H\rto^-{q_{H}}\dto^{a'} & S^{n-1}\rto^-{\ast}\dto^{f} 
          & S^{m}\ddouble \\ 
      \Omega S^{m}\rto^-{\partial_{K}} & K\rto^-{q_{K}} & \overline{M}\rto^-{\ast} & S^{m} 
  \enddiagram\] 
then choosing a right homotopy inverse 
\(t_{M}\colon\namedright{\overline{M}}{}{K}\) 
for $q_{K}$ and pulling back to give a right homotopy inverse 
\(t_{H}\colon\namedright{S^{n-1}}{}{H}\) 
for $q_{H}$, the naturality of the homotopy actions implies there is a homotopy commutative diagram 
\[\diagram 
      S^{n-1}\times\Omega S^{m}\rto^-{t_{H}\times 1}\dto^{f\times 1} 
         & H\times\Omega S^{m}\rto^-{\vartheta_{H}}\dto^{a'\times 1} & H\dto^{a'} \\ 
      \overline{M}\times\Omega S^{m}\rto^-{t_{M}\times 1} & K\times\Omega S^{m}\rto^-{\vartheta_{K}} 
         & K 
  \enddiagram\] 
in which both rows are homotopy equivalences. 

Let $e=\vartheta_{H}\circ(t_{H}\times 1)$ be the homotopy equivalence for $H$ along the upper row. 
Observe that 
\[q_{H}\circ e=q_{H}\circ\vartheta_{H}\circ(t_{H}\times 1)\simeq q_{H}\circ\pi_{1}\circ (t_{H}\times 1) 
    \simeq q_{H}\circ t_{H}\circ\pi_{1}\simeq\pi_{1}\]  
where the first equality is the definition of $e$, the second holds because 
$q_{H}\circ\vartheta_{H}\simeq q_{H}\circ\pi_{1}$ by a standard property of the homotopy action, the third 
holds since~$\pi_{1}$ is natural, and the fourth holds since $t_{H}$ is a right homotopy inverse 
for $q_{H}$. 

Turning now to the left face of the cube, observe that there is a homotopy fibration diagram 
\begin{equation} 
  \label{HGfibdgrm} 
  \diagram 
      \Omega S^{m}\rto^-{\partial_{H}}\ddouble & H\rto^<<<<{q_H}\dto^{b'} & S^{n-1}\rto^-{\ast}\dto^{g} 
           & S^{m}\ddouble \\ 
      \Omega S^{m}\rto^-{\partial_{G}} & G\rto^<<<<{q_G} & \overline{T}\rto^-{\overline{h}} & S^{m} 
  \enddiagram 
\end{equation} 
where $b'$ is an induced map of fibres. The middle square is the left face of the cube. 
Define the map~$\mathfrak{g}$ by the composite 
\[\mathfrak{g}\colon\nameddright{S^{n-1}}{t_{H}}{H}{b'}{G}.\] 
Consider the diagram 
\[\diagram 
    S^{n-1}\times\Omega S^{m}\rto^-{t_{H}\times 1}\drto_{\mathfrak{g}\times 1}  
        & H\times\Omega S^{m}\rto^-{\vartheta_{H}}\dto^{b'\times 1} & H\dto^{b'} \\ 
     & G\times\Omega S^{m}\rto^-{\vartheta_{G}} & G. 
  \enddiagram\] 
The triangle commutes by definition of $\mathfrak{g}$ and the square homotopy commutes 
by the naturality of the homotopy action. The top row is the definition of the homotopy equivalence $e$. 
The homotopy commutativity of the diagram therefore says that 
$b'\circ e\simeq\vartheta_{G}\circ(\mathfrak{g}\times 1)$. 

Summarizing, if the homotopy equivalence $e$ is used to replace $H$ in the middle square of~(\ref{HGfibdgrm}), then as $q_{H}\circ e\simeq\pi_{1}$ and 
$b'\circ e\simeq \vartheta_{G}\circ(\mathfrak{g}\times 1)$, we obtain a homotopy commutative square 
\begin{equation} 
  \label{HGfibdgrmreplace} 
  \diagram 
    S^{n-1}\times\Omega S^{m}\dto^{\vartheta_{G}\circ(\mathfrak{g}\times 1)}\rto^-{\pi_{1}} 
        & S^{n-1}\dto^{g} \\ 
    G\rto^-{q_{G}} & \overline{T}. 
  \enddiagram 
\end{equation}  
As the middle square in~(\ref{HGfibdgrm}) is a homotopy pullback, so is~(\ref{HGfibdgrmreplace}) 
since $e$ is a homotopy equivalence. As the middle square in~(\ref{HGfibdgrm}) is the left face 
in the cube~(\ref{conncube}), we may now replace it with~(\ref{HGfibdgrmreplace}), showing 
that~$b\simeq\vartheta_{G}\circ (\mathfrak{g}\times 1)$. 
\end{proof}

The top face of~\eqref{conncube0} can be refined. Consider the homotopy fibration sequence 
\(\namedddright{\Omega S^{m}}{\partial_{G}}{G}{}{\overline{T}}{\overline{h}}{S^{m}}\). 
Since~$\overline{h}$ has a right inverse, the connecting map $\partial_{G}$ is null homotopic. 
As the restriction of the homotopy action $\vartheta_{G}$ to $\Omega S^{m}$ is $\partial_{G}$, 
the null homotopy for $\partial_{G}$ implies that $\vartheta_{G}$ factors as a composite 
\begin{equation}\label{thetaGeq}
\vartheta_G: \ \ \nameddright{G\times\Omega S^{m}}{q}{G\rtimes\Omega S^{m}}{\theta_{G}}{G},
\end{equation} 
where $q$ collapses $\Omega S^{m}$ to a point and $\theta_{G}$ is an induced quotient map. 
There may be a choice in the quotient map $\theta_G$ that factors $\vartheta_G$. 
Any choice of $\theta_{G}$ satsifies the following. 

\begin{lemma} 
  \label{thetapo} 
 There is a homotopy pushout
 \[\diagram 
       S^{n-1}\rtimes\Omega S^{m}\rto^-{f\rtimes 1}\dto^{\theta_{G}\circ(\mathfrak{g}\rtimes 1)} 
            & \overline{M}\rtimes\Omega S^{m}\dto \\ 
       G\rto & F. 
  \enddiagram\]  
\end{lemma} 

\begin{proof}
Consider the diagram 
\[
\diagram
S^{n-1}\times \Omega S^m \rto^{q}   \dto^{f\times 1} & S^{n-1}\rtimes \Omega S^m \dto^{f\rtimes 1} \rto^{\mathfrak{g}\rtimes 1}  & G\rtimes \Omega S^m \rto^<<<{\theta_G}  & G \dto^{} \\
 \overline{M}\times \Omega S^m \rto^{q}  & \overline{M}\rtimes \Omega S^m \rrto^{}  && F'.
\enddiagram
\] 
The left square commutes by the naturality of $q$ and is a homotopy pushout because a 
common factor of $\Omega S^{m}$ has been collapsed out from the left side. The right square 
defines $F'$ as the homotopy pushout of $f\rtimes 1$ and $\vartheta_{G}\circ(\mathfrak{g}\rtimes 1)$. 
As both squares are homotopy pushouts, so is the outer rectangle. Along the top 
row, the naturality of $q$ and the factorization 
$\vartheta_{G}\simeq\theta_{G}\circ q$ implies that 
\[\vartheta_{G}\circ(\mathfrak{g}\rtimes 1)\circ q=\vartheta_{G}\circ q\circ(\mathfrak{g}\times 1)\simeq  
     \vartheta_{G}\circ(\mathfrak{g}\times 1).\] 
Therefore the outer rectangle is the homotopy pushout of $f\times 1$ and 
$\vartheta_{G}\circ(\mathfrak{g}\times 1)$, implying that $F'\simeq F$ by top face of \eqref{conncube0}. Therefore 
$F$ is the homotopy pushout of $f\rtimes 1$ and $\theta_{G}\circ(\mathfrak{g}\rtimes 1)$, as asserted. 
\end{proof}

Lemma \ref{thetapo} lets us prove Theorem \ref{fibid}.
\begin{proof}[Proof of Theorem \ref{fibid}]
By definition of $F$, there is a homotopy fibration
\[\nameddright{F}{}{M\conn T}{h'}{S^{m}}.\] 
By Lemma \ref{h'lemmaT} the map $h'$ has a right homotopy inverse, implying 
that there is a homotopy equivalence $\Omega(M\conn T)\simeq\Omega S^{m}\times\Omega F$. 
To complete the proof of the theorem, it remains to identify $F$ with 
$E\vee (\overline{M}\rtimes\Omega S^{m})$.

Consider the diagram of ``data" 
\[
\diagram 
       & G\rto^{\jmath}\dto^{q_G} & E\dto^{q_E} \\ 
       S^{n-1}\rto^-{g} & \overline{T}\rto^{j}\dto^{\overline{h}} & T\dto^{h} \\ 
       & S^m\rdouble & S^m, 
  \enddiagram
\]
where the middle row is the homotopy cofibration, the two columns are homotopy fibrations that 
define $G$ and $E$, and the map from the left column to the right is a homotopy 
fibration diagram. In the proof of Lemma \ref{conncubelemma}, we chose a lift $\mathfrak{g}=b'\circ t_H$ of $g$. We claim that it is consistent with the choice of the lift of $g$ in Remark \ref{BTremark}. 
Indeed, there is a homotopy commutative diagram
\[
\diagram
S^{n-1} \rto^{t_H} \drdouble  & H \rto^{b'} \dto^{q_H}  & G\rto^{\jmath}  \dto^{q_G}  & E \dto^{q_E} \\
 &                           S^{n-1} \rto^{g}          &\overline{T} \rto^{j}         & T,          
\enddiagram
\]
where the right square is part of the ``data'', the middle square is the middle square of \eqref{HGfibdgrm}, and $t_H$ is a right homotopy inverse of $q_H$ defined in the proof of Lemma \ref{conncubelemma}. Since $j\circ g$ is null homotopic, so is the composite $q_E\circ\jmath\circ b'\circ t_H$ by the homotopy commutativity of the diagram. It follows that $\jmath\circ b'\circ t_H$ can be lifted to a map through the connecting map $\Omega S^m\stackrel{\partial_E}{\larrow} E$. However, $h$ having a right homotopy inverse implies that $\partial_E$ is null homotopic, and therefore so is $\jmath\circ b'\circ t_H$. Thus the lift $\mathfrak{g}=b'\circ t_H$ can be factored through the homotopy fibre of $G\stackrel{\jmath}{\larrow}E$, which is also the homotopy fibre of $\overline{T}\stackrel{j}{\larrow} T$. 
This implies that $\mathfrak{g}$ satisfies the requirements for the lift of $g$ in Remark \ref{BTremark}.

With the chosen lift $\mathfrak{g}$ of $g$, by Theorem \ref{BT} and Remark \ref{BTremark}, there is a choice 
of $\theta_{G}$ such that there is a homotopy cofibration 
\[\llnameddright{S^{n-1}\rtimes\Omega S^{m}}{\theta_{G}\circ(\mathfrak{g}\rtimes 1)}{G}{}{E}.\] 
By \cite[Proposition 5.1]{T3}, the hypotheses that $T$ is an $(m-1)$-connected, $n$-dimensional Poincar\'{e} 
Duality complex with $h$ having a right homotopy inverse implies that 
$\theta_{G}\circ(\mathfrak{g}\rtimes 1)$ has a left homotopy inverse and there is a homotopy equivalence 
\[G\stackrel{e}{\longrightarrow} (S^{n-1}\rtimes\Omega S^{m})\vee E\] 
with the property that $e\circ\theta_{G}\circ(\mathfrak{g}\rtimes 1)$ is homotopic 
to the inclusion $i_{1}$ of the first wedge summand. 
Thus the homotopy pushout in Lemma~\ref{thetapo} can be rewritten, up to homotopy equivalence, as a homotopy pushout 
\[\diagram 
     S^{n-1}\rtimes\Omega S^{m}\rto^-{f\rtimes 1}\dto^{i_{1}} & \overline{M}\rtimes\Omega S^{m}\dto \\ 
     (S^{n-1}\rtimes\Omega S^{m})\vee E\rto & F. 
  \enddiagram\] 
In general, if $Q$ is defined by the homotopy pushout 
\[\diagram 
      A\rto\dto^{i_{1}} & B\dto \\
      A\vee C\rto & Q 
  \enddiagram\] 
then $Q\simeq B\vee C$. 
Therefore, in our case, there is a homotopy equivalence 
$F\simeq (\overline{M}\rtimes\Omega S^{m})\vee E$. 
\end{proof}

\section{Stabilizing by other $T$, Part II} 
\label{sec:Tstab2} 
In this section, an alternative proof of Theorem \ref{fibid} is given when $M$ is simply connected. 
We will sometimes refer to the constructions and notations in Section \ref{sec:Tstab}.

The strategy is to apply Theorem \ref{BT} to two suitable diagrams of ``data" and compare them using the naturality property of Theorem \ref{BT} stated in Remark \ref{BTnatremark}. Based on the constructions in Section~\ref{sec:Tstab}, there are two diagrams
\[
\diagram 
       & G\rto\dto & E\dto          &&      
       & F \rto \dto & E \dto \\ 
       S^{n-1}\rto^-{g} & \overline{T}\rto^{j}\dto^{\overline{h}} & T\dto^{h}    &&
       \overline{M} \rto^{}  & M\conn T  \rto^{} \dto^{h'}   & T \dto^{h}\\
       & S^m\rdouble & S^m,   && 
       & S^m\rdouble & S^m.
         \enddiagram
\]
The left diagram was already used in the first proof of Theorem \ref{fibid}, in which the middle row is the homotopy cofibration for the top cell attachment of $T$. In the right diagram, the middle row is the homotopy cofibration obtained by collapsing $\overline{M}$ in $M\conn T$ to a point. In both 
diagrams the columns are homotopy fibrations, where the map $h$ is given by assumption 
and $\overline{h}$ and $h'$ are defined as the composites $\overline{T}\stackrel{j}{\larrow} T\stackrel{h}{\larrow} S^m$ and $M\conn T\stackrel{}{\larrow} T\stackrel{h}{\larrow} S^m$ respectively. 
In each case, the columns form a homotopy fibration diagram.  

We claim that both $\overline{h}$ and $h'$ have a right homotopy inverse. Indeed, by assumption the map $h$ has a right homotopy inverse $s$. For dimension reasons, the map 
\(\namedright{S^{m}}{s}{T}\) 
factors as a composite 
\[\nameddright{S^{m}}{\overline{s}}{\overline{T}}{j}{T}\] 
for some map $\overline{s}$. Thus $\overline{h}\circ\overline{s}=h\circ j\circ\overline{s}\simeq h\circ s$, 
which is homotopic to the identity map since $s$ is a right homotopy inverse for $h$. 
Therefore $\overline{s}$ is a right homotopy inverse of $\overline{h}$. Further, let $i_T: \overline{T}\larrow M\conn T$ be the canonical inclusion. By~\eqref{connpo}, $\overline{h}\simeq h'\circ i_T$. Therefore 
$h'\circ i_T\circ  \overline{s}\simeq \overline{h}\circ\overline{s}$ is homotopic to the 
identity map, that is, $i_T\circ  \overline{s}$ is a right homotopy inverse of $h'$. Hence 
Theorem \ref{BT} can be applied to obtain homotopy cofibrations
\[
S^{n-1}\rtimes\Omega S^{m} \stackrel{\Gamma}{\larrow} G\larrow E\qquad\mbox{and}\qquad 
\overline{M}\rtimes\Omega S^{m} \stackrel{\Gamma'}{\larrow} F\larrow E. 
\] 

We wish to apply the naturality property in Remark~\ref{BTremark}. To do so, observe that 
as $\overline{h}\simeq h'\circ i_T$ there is a homotopy fibration diagram
\[
\diagram 
G \rto \dto^{}  & \overline{T} \rto^{\overline{h}} \dto^{i_T} & S^{m}\ddouble \\
F \rto             &  M\conn T   \rto^{h'}                             & S^m.
\enddiagram
\]
Observe also that a right homotopy inverse of $h'$ is $i_{T}\circ\overline{s}$, where $\overline{s}$ is a right homotopy inverse for $\overline{h}$, 
so the right homotopy inverses satisfy a homotopy commutative diagram 
\[\diagram 
     S^{m}\rto^-{\overline{s}}\dto^{=} & \overline{T}\dto^{i_{T}} \\ 
     S^{m}\rto^-{i_{T}\circ\overline{s}} & M\conn T. 
 \enddiagram\] 
Note this is stronger than the compatibility required in~(\ref{ss'diag}) since it occurs before looping.  
Hence, by Remark \ref{BTnatremark}, there is a homotopy cofibration diagram
\begin{equation}
\label{GFdiag}
\diagram
S^{n-1}\rtimes\Omega S^{m} \rto^<<<{\Gamma} \dto & G\rto^{} \dto^{}  & E \ddouble \\
\overline{M}\rtimes\Omega S^{m}   \rto^<<<<{\Gamma'}  & F\rto^{}  & E.
\enddiagram
\end{equation}

Checking connectivities, as $m\geq 2$ the space $\Omega S^{m}$ is path-connected, so as $n\geq 2$ 
the space $S^{n-1}\rtimes\Omega S^{m}$ is simply-connected. Since $T$ is $(m-1)$-connected 
with $m\geq 2$ and $h$ has a right homotopy inverse, the homotopy fibre $E$ of $h$ is also 
simply-connected. Similarly, $G$ is simply-connected. Since $M$ is assumed to be simply-connected, 
so is $M\conn T$, and arguing as for $E$ shows that $F$ is simply-connected. Thus all spaces 
in the homotopy cofibration diagram~(\ref{GFdiag}) are simply-connected. 

By \cite[Proposition 5.1]{T3}, the hypotheses that $T$ is an $(m-1)$-connected, $n$-dimensional Poincar\'{e} 
Duality complex with $h$ having a right homotopy inverse implies that $\Gamma$ has 
a left homotopy inverse and there is a homotopy equivalence 
\[G\simeq (S^{n-1}\rtimes\Omega S^{m})\vee E.\] 
In particular, the map $G\larrow E$ has a right homotopy inverse, and so does the map $F\larrow E$ by the homotopy commutativity of the right square of Diagram \eqref{GFdiag}. Let $t: E\larrow F$ be such a right homotopy inverse. Then the lower homotopy cofibration in Diagram \eqref{GFdiag} implies that the composite
\[
\Phi: (\overline{M}\rtimes\Omega S^{m})\vee E\stackrel{\Gamma'\vee t}{\larrow} F\vee F \stackrel{\nabla}{\larrow} F
\]
induces an isomorphism on homology, where $\nabla$ is the folding map. Since $F$, $E$ and $\overline{M}$ are simply-connected, Whitehead's Theorem implies that $\Phi$ is a homotopy equivalence. 

To conclude, the homotopy fibration 
\(\nameddright{F}{}{M\conn T}{h'}{S^{m}}\) 
defining $F$ can be rewritten as a homotopy fibration  
\[\nameddright{E\vee (\overline{M}\rtimes\Omega S^{m})}{}{M\conn T}{h'}{S^{m}},\]  
where $h'$ has a right homotopy inverse. This  re-proves Theorem \ref{fibid} when $M$ is simply connected.

\section{Stabilizing by a product of spheres} 
\label{sec:spherestab} 
There is an especially interesting case of Theorem~\ref{fibid}. Take $T=S^{m}\times S^{n-m}$ 
for $2\leq m\leq n-m$ and take 
\(\namedright{S^{m}\times S^{n-m}}{h}{S^{m}}\) 
as the projection onto the first factor. Then the inclusion 
\(\namedright{S^{m}}{}{S^{m}\times S^{n-m}}\) 
of the first factor is a right homotopy inverse for $h$ and the homotopy fibre of~$h$ is $S^{n-m}$. 
Theorem~\ref{fibid} therefore implies the following. 

\begin{theorem} 
   \label{sphereprodstabilize} 
   Let $M$ be a path-connected $n$-dimensional Poincar\'{e} Duality complex with a single top cell.  
   If $2\leq m\leq n-m$ then there is a homotopy fibration 
   \[\nameddright{S^{n-m}\vee(\overline{M}\rtimes\Omega S^{m})}{}{M\conn(S^{m}\times S^{n-m})}{h'}{S^{m}}\] 
   where $h'$ has a right homotopy inverse.~$\qqed$ 
\end{theorem}

\begin{example} 
Let $M=\mathbb{R}P^{4}$. Then $\overline{M}=\mathbb{R}P^{3}$ and by 
Theorem~\ref{sphereprodstabilize} there is a homotopy fibration 
\[\nameddright{S^{2}\vee(\mathbb{R}P^{3}\rtimes\Omega S^{2})}{}{\mathbb{R}P^{4}\conn(S^{2}\times S^{2})} 
     {h'}{S^{2}}\] 
where $h'$ has a right homotopy inverse. 
\end{example} 

\begin{example} 
\label{freepi1} 
Let $M$ be a path-connected closed $4$-manifold with free fundamental group. By~\cite{KM}, 
$\overline{M}$ is homotopy equivalent to a wedge $W$ of $1$, $2$ and $3$-dimensional spheres. 
Therefore by Theorem~\ref{sphereprodstabilize} there is a homotopy fibration 
\[\nameddright{S^{2}\vee(W\rtimes\Omega S^{2})}{}{M\conn(S^{2}\times S^{2})}{h'}{S^{2}}\] 
where $h'$ has a right homotopy inverse. Going further, in general, if $A$ is a path-connected 
co-$H$-space then for any path-connected space $B$ there is a homotopy equivalence 
$A\rtimes B\simeq A\vee(A\wedge B)$. In our case, as $W$ is a wedge of path-connected spheres 
it is a co-$H$-space, so $W\rtimes\Omega S^{2}\simeq W\vee (W\wedge\Omega S^{2})$. By 
James~\cite{J}, there is a homotopy equivalence 
\[\Sigma\Omega S^{2}\simeq\bigvee_{k=1}^{\infty} S^{k+1}.\] 
Therefore, for each copy of $S^{t}$ in $W$, for $1\leq t\leq 3$, the space $S^{t}\wedge\Omega S^{2}$ 
is homotopy equivalent to a wedge of spheres. Thus $W\wedge\Omega S^{2}$ 
is homotopy equivalent to a wedge of spheres. Consequently, the space 
$S^{2}\vee(W\rtimes\Omega S^{2})$ is homotopy equivalent to a wedge of spheres. 
\end{example} 

An interesting example in the simply-connected case is the Wu manifold $SU(3)/SO(3)$. 
For $m\geq 2$, let $P^{m}(2)$ be the \emph{mod-$2$ Moore space} obtained as the cofibre 
of the degree $2$ map on $S^{m-1}$. Note that there is a homotopy equivalence 
$P^{m+1}(2)\simeq\Sigma P^{m}(2)$. 

\begin{example} 
Let $M=SU(3)/SO(3)$. As a $CW$-complex, $M=P^{3}(2)\cup e^{5}$. Thus 
$\overline{M}=P^{3}(2)$. By Theorem~\ref{sphereprodstabilize} there 
is a homotopy fibration 
\[\nameddright{S^{3}\vee (P^{3}(2)\rtimes\Omega S^{2})}{}{M\conn(S^{2}\times S^{3})}{h'}{S^{2}}\] 
where $h'$ has a right homotopy inverse. Since $P^{3}(2)\simeq\Sigma P^{2}(2)$, arguing as in 
Example~\ref{freepi1}, there are homotopy equivalences 
$P^{3}(2)\rtimes\Omega S^{2}\simeq P^{3}(2)\vee (P^{3}(2)\wedge\Omega S^{2})$ and 
\[P^{3}(2)\wedge\Omega S^{2}\simeq\bigvee_{k=1}^{\infty} P^{2}(2)\wedge S^{k+1}\simeq 
     \bigvee_{k=1}^{\infty} P^{k+3}(2).\] 
Thus there is a homotopy fibration 
\[\nameddright{S^{3}\vee\big(\bigvee_{k=1}^{\infty} P^{k+2}(2)\big)}{}{M\conn(S^{2}\times S^{3})}{h'}{S^{2}}.\] 
This is interesting because little is known about the homotopy groups of $M$. Yet after 
stabilizing with $S^{2}\times S^{3}$ the situation becomes much more tractable. 
\end{example}

\bibliographystyle{amsalpha}

\end{document}